\documentclass[11pt,twoside,english]{myamsart}
\normalsize
\advance\oddsidemargin by -1.7cm
\advance\evensidemargin by -1.7cm
\advance\topmargin by -1.0cm 
\usepackage{longtable}
\usepackage{srcltx} 
\usepackage{color}
\usepackage{amssymb}
\usepackage{babel}
\usepackage{amsmath}
\usepackage{amscd}   
\usepackage{epsfig}
\usepackage{rotating}
\usepackage{psfrag}
\usepackage{multirow}
\usepackage{arydshln}
\usepackage{float}
\usepackage{enumerate}
\let\mathg\mathfrak
\theoremstyle{plain}
\newtheorem{cor}{Corollary}[section]
\newtheorem{lem}[cor]{Lemma}
\newtheorem{thm}[cor]{Theorem}
\newtheorem{prop}[cor]{Proposition}

\newtheoremstyle{thmstylenn}
{15pt}
{5pt}
{\it}
{}
{\bf}
{ \ref{lem:sp2rep}.}
{ }
{}
\theoremstyle{thmstylenn}

\theoremstyle{definition}
\newtheorem{exa}[cor]{Example}
\newtheorem{NB}[cor]{Remark}
\newtheorem{NBs}[cor]{Remarks}
\newtheorem{dfn}[cor]{Definition}
\newtheorem{nota}[cor]{Notation}

\numberwithin{table}{section}
\numberwithin{equation}{section}
\renewcommand{\labelenumi}{$(\theenumi)$}
\newcommand{\Kommentar}[1]{}
\newcommand{\Kommentars}[1]{} 

\newcommand{\bdm}{\begin{displaymath}}
\newcommand{\edm}{\end{displaymath}}
\newcommand{\be}{\begin{equation}}
\newcommand{\ee}{\end{equation}}
\newcommand{\bea}[1][]{\begin{eqnarray#1}}
\newcommand{\eea}[1][]{\end{eqnarray#1}}
\newcommand{\btab}{\begin{tabular}}
\newcommand{\etab}{\end{tabular}}

\usepackage{rotating}

\newcommand{\ra}{\rightarrow}
\newcommand{\lra}{\longrightarrow}

\newcommand{\lan}{\left\langle}
\newcommand{\ran}{\right\rangle}

\newcommand{\Id}{\ensuremath{\mathrm{Id}}}
\newcommand{\tr}{\ensuremath{\mathrm{tr}}}

\newcommand{\del}{\partial}

\newcommand{\cyclic}[1]{\stackrel{{\scriptsize #1}}{\mathfrak{S}}}
\newcommand{\C}{\ensuremath{\mathbb{C}}}

\newcommand{\R}{\ensuremath{\mathbb{R}}}

\newcommand{\Z}{\ensuremath{\mathbb{Z}}}

\renewcommand{\P}{\ensuremath{\mathbb{P}}}

\newcommand{\CP}{\ensuremath{\mathbb{CP}}}
\newcommand{\RP}{\ensuremath{\mathbb{RP}}}

\newcommand{\eps}{\ensuremath{\varepsilon}}

\newcommand{\End}{\ensuremath{\mathrm{End}}}

\newcommand{\Scal}{\ensuremath{\mathrm{Scal}}}

\renewcommand{\Im}{\ensuremath{\mathrm{Im\,}}}
\renewcommand{\Re}{\ensuremath{\mathrm{Re\,}}}

\newcommand{\diag}{\ensuremath{\mathrm{diag}}}

\newcommand{\GL}{\ensuremath{\mathrm{GL}}}

\newcommand{\SL}{\ensuremath{\mathrm{SL}}}

\newcommand{\su}{\ensuremath{\mathg{su}}}
\newcommand{\SU}{\ensuremath{\mathrm{SU}}}
\newcommand{\G}{\ensuremath{\mathrm{G_2}}}
\newcommand{\U}{\ensuremath{\mathrm{U}}}

\newcommand{\so}{\ensuremath{\mathg{so}}}
\renewcommand{\sl}{\ensuremath{\mathg{sl}}}
\newcommand{\SO}{\ensuremath{\mathrm{SO}}}

\newcommand{\Spin}{\ensuremath{\mathrm{Spin}}}

\newcommand{\g}{\ensuremath{\mathfrak{g}}}
\newcommand{\h}{\ensuremath{\mathfrak{h}}}
\newcommand{\m}{\ensuremath{\mathfrak{m}}}

\newcommand{\WG}{\ensuremath{\mathcal{W}}}
\newcommand{\WS}{\chi}
\def\sideremark#1{\ifvmode\leavevmode\fi\vadjust{\vbox to0pt{\vss
\hbox to 0pt{\hskip\hsize\hskip1em%
\vbox{\hsize2cm\tiny\raggedright\pretolerance10000%
\noindent {\color{red}{\bf #1}}\hfill}\hss}\vbox to8pt{\vfil}\vss}}}%
\def\w{\wedge}

\newcommand{\lto}{\longrightarrow}
\newcommand{\ltto}{\longmapsto}
\newcommand{\tto}{\mapsto}
\def\iso{\cong}
\newcommand{\hook}{\ensuremath{\lrcorner\,}}

\newcommand{\ol}{\overline}
\newcommand{\tx}{\textstyle}
\def\tsum{\tx\sum}
\newcommand{\q}{\quad}\newcommand{\qq}{\qquad}
\newcommand{\ba}{\begin{array}}\newcommand{\ea}{\end{array}}

\renewcommand{\leq}{\leqslant}
\newcommand{\Ker}{\ensuremath{\mathrm{Ker\,}}}

\newcommand{\cliffmult}{m}
\newcommand{\phis}{\phi} 
\newcommand{\Sigmas}{\Sigma} 
\newcommand{\jphis}{j(\phis)} 
\newcommand{\psis}{\psi_{\phi}} 
\newcommand{\psims}{\psis^J} 
\newcommand{\J}{J_{\phis}} 
\newcommand{\js}{j} 
\newcommand{\phib}{\phi^*} 
\newcommand{\f}{\Psi_{\phi}} 
\newcommand{\phig}{\phi} 

\newcommand{\phish}{\phi} 
\newcommand{\phigh}{\phi} 
\newcommand{\phigc}{\bar\phis}
\newcommand{\phisc}{\phis}
\newcommand{\jphisc}{\jphis}
\newcommand{\phigk}{\phis}
\newcommand{\phisk}{\phis}

\newcommand{\esomorphism}{\psi}

\newtheorem*{acknowledgements}{Acknowledgements}
\begin{document}
\def\haken{\mathbin{\hbox to 6pt{%
                 \vrule height0.4pt width5pt depth0pt
                 \kern-.4pt
                 \vrule height6pt width0.4pt depth0pt\hss}}}
    \let \lrcorner\haken
\setcounter{equation}{0}
%
%
\thispagestyle{empty}
%
\date{\today}
\title[Spinorial description of $\SU(3)$- and $\G$-manifolds]{Spinorial description of $\SU(3)$- and $\G$-manifolds}

%
%
%
\author[I.Agricola]{Ilka Agricola}
\author[S.Chiossi]{Simon G. Chiossi}
\author[T.Friedrich]{Thomas Friedrich}
\author[J.H\"oll]{Jos H\"oll}
%
%
\address{\hspace{-7mm} 
Ilka Agricola, Jos H\"oll,
Fachbereich Mathematik und Informatik,
Philipps-Universit\"at Marburg,
Hans-Meerwein-Strasse,
35032 Marburg, Germany\newline
{\normalfont\ttfamily agricola@mathematik.uni-marburg.de}, 
{\normalfont\ttfamily hoellj@mathematik.uni-marburg.de}}
\address{\hspace{-7mm} 
Simon G. Chiossi,
Departamento de Matematica,
Universidade Federal da Bahia,
Av. Adhemar de Barros s/n, Ondina,
40170-110 Salvador/BA, Brazil\newline
{\normalfont\ttfamily simon.chiossi@polito.it}}
\address{\hspace{-7mm} 
Thomas Friedrich,
Institut f\"ur Mathematik,
Humboldt-Universit\"at zu Berlin,
10099 Berlin, Germany\newline
{\normalfont\ttfamily friedric@mathematik.hu-berlin.de}}
%
%
\begin{abstract}
We present a uniform  description of 
$\SU(3)$-structures in dimension
$6$ as well as $\G$-structures in dimension $7$ in terms of
a characterising spinor and the spinorial field equations it satisfies.
We apply the results to hypersurface theory to obtain new
embedding theorems, and give a general recipe for building 
conical manifolds. The approach also enables one to 
subsume all variations of the notion of a Killing spinor. 
%
\end{abstract}
\maketitle
\pagestyle{headings}
\frenchspacing
\section{Introduction}
%
This paper is devoted to a systematic and uniform description of 
$\SU(3)$-structures in dimension
$6$, as well as $\G$-structures in dimension $7$, using a spinorial formalism.
Any $\SU(3)$- or $\G$-manifold can be understood as a 
Riemannian spin manifold of dimension $6$ or $7$, respectively, equipped 
with a real spinor field $\phisc$ or $\phigc$ of length one. 
Let us denote by $\nabla$ the Levi-Civita connection and its lift
to the spinor bundle. We prove that an $\SU(3)$-manifold admits a $1$-form 
$\eta$ and an endomorphism field $S$ such that
the spinor $\phisc$ solves, for any vector field $X$,
\bdm
\nabla_X \phisc=\eta(X) \jphisc+S(X)\cdot \phisc,
\edm
where $j$ is the $\Spin(6)$-invariant complex structure on the spin
representation space $\Delta=\R^8$ realising the isomorphism 
$\Spin(6)\cong\SU(4)$. In a similar vein, there exists an endomorphism $\bar S$
such that the spinor $\phigc$ of a $\G$-manifold satisfies 
the even simpler equation
\bdm
\nabla_X \phigc = \bar S (X)\cdot \phigc.
\edm
We identify the characterising entities $\eta,\, S$, and $\bar S$
with certain components of the intrinsic torsion and use them to describe 
the basic classes of $\SU(3)$- and $\G$-manifolds by means of a spinorial 
field equation.
For example, it is known that nearly K\"ahler manifolds
correspond to $S=\mu\, \Id$ and $\eta=0$ \cite{Grunewald90}, 
and nearly parallel 
$\G$-manifolds are those with $\bar S=\lambda\,\Id$ \cite{FKMS},
since the defining equation reduces then to the classical constraint for a 
Riemannian Killing spinor. 
If $S$ or $\bar S$ is symmetric (and, in dimension $6$, 
additionally $\eta=0$), this is the equation defining 
generalised Killing spinors, which are known to correspond
to half-flat structures \cite{CS02} 
and cocalibrated $\G$-structures \cite{CS06}. 
For all other classes, Theorems \ref{thm:class} and \ref{G2thm} 
provide new information concerning 
$\phisc$ and $\phigc$. To mention but one example, we shall characterise 
in Theorem \ref{thm:Dphisu3} 
 Riemannian spin $6$-manifolds admitting a harmonic spinor of
constant length. 
Theorem \ref{G2thm} states the analogue fact for $\G$-manifolds.
\bigbreak

We begin by reviewing algebraic aspects of the dimensions $6$ and 
$7$---and explain why it is more convenient to use, in the former case, 
real spinors instead of complex spinors. In section \ref{sec:tysu3} we  
carefully relate the various geometric quantities cropping up in special
Hermitian geometry, with particular care regarding: the vanishing or 
(anti-)symmetry of $S,\eta$, the intrinsic torsion, induced differential forms and 
Nijenhuis tensor,  Lee and K\"ahler forms,  
and the precise spinorial PDE for $\phi$. We introduce 
a connection well suited to describe the geometry, 
and its relationship to the more familiar characteristic connection.
The same programme is then carried out in section \ref{sec:g2} for $\G$-manifolds.
The first major application of this set-up occupies section \ref{sec:subman}: 
 our results can be used to study embeddings of $\SU(3)$-manifolds in
$\G$-manifolds and describe different types of cones (section \ref{sec:spincones}).
The latter results complement the first and last author's work
\cite{AH13}. 
This leads to the inception of a more unified picture 
relating the host of special spinor fields occuring in different
parts of the literature: Riemannian Killing spinors, generalised
Killing spinors, quasi-Killing spinors, Killing spinors with torsion 
etc. What we show in section \ref{sec:KswT} is that all those turn out 
to be special instances of the characterising spinor field equations for 
$\phi$ and $\phigc$ that we started with, and although looking, in general,
quite different, these equations can be drastically simplified in specific  
situations.\bigbreak

The pattern that emerges here clearly indicates that the spinorial approach 
is not merely the overhaul of an established theory. Our point is precisely that it 
should be used to describe efficiently these and other types of geometries, 
like $\SU(2)$- or $\Spin(7)$-manifolds, and that it provides more 
information than previously known.
Additionally, the explicit formulas furnish a working toolkit for understanding 
many different concrete examples, and for further study.
%
%
\section{Spin linear algebra}\label{sec:linal}
%
%
The real Clifford algebras in dimensions $6,7$ are isomorphic to 
End$(\R^8)$ and End$(\R^8) \oplus$ End$(\R^8)$ respectively. The spin 
representations are real and $8$-dimensional, so they coincide as 
vector spaces, and we denote this common space by 
$\Delta := \mathbb{R}^8$. By fixing an orthonormal basis $e_1, ... ,e_7$ of the 
Euclidean  space $\mathbb{R}^7$, one choice for the real representation of 
the Clifford algebra on $\Delta$ is \label{SU3basis}
\begin{align*}
 e_{1} &=  +E_{18} + E_{27} - E_{36} - E_{45}, \qquad
e_{2} =-E_{17} + E_{28} + E_{35} - E_{46},\\
e_{3} &= -E_{16} + E_{25} - E_{38} + E_{47}, \qquad
e_{4} = -E_{15} - E_{26} - E_{37} - E_{48},\\
e_{5} &= -E_{13} - E_{24} + E_{57} + E_{68}, \qquad
e_{6} =+E_{14} - E_{23} - E_{58} + E_{67},\\
e_{7} &=+E_{12} - E_{34} - E_{56} + E_{78},
\end{align*}
where the matrices $E_{ij}$ denote the standard basis elements of the 
Lie algebra 
$\so(8)$, i.\,e. the endomorphisms mapping $e_i$ to $e_j$, $e_j$ to $-e_i$ and 
everything else to zero.\smallbreak

We begin by discussing the $6$-dimensional case. Albeit real, the 
spin representation 
 admits a $\Spin(6)$-invariant complex structure 
$\js : \Delta \rightarrow \Delta$ defined be the formula
\bdm
\js \ := \ e_1 \cdot e_2 \cdot e_3 \cdot e_4 \cdot e_5 \cdot e_6  .
\edm 
Indeed, $\js^2 = -1$ and $\js$ anti-commutes with the Clifford 
multiplication by vectors of $\R^6$;
this reflects the fact that $\Spin(6)$ is isomorphic to $\SU(4)$.
The complexification of $\Delta$ splits,
\bdm
\Delta \otimes_{\R} \C \ = \ \Delta^+ \, \oplus \Delta^- ,
\edm 
a consequence of the fact that $j$ is a real structure making $(\Delta,j)$ 
 complex-(anti)-linearly isomorphic to either $\Delta^\pm$, via 
$\phi \tto \phi \, \pm \ i \cdot \js(\phi)$.
Any real spinor $0 \not= \phi \in \Delta$, furthermore, decomposes  $\Delta$ into three pieces,
\be\label{eq:splittingsu3}
\Delta \ = \ \R\phi  \oplus  \R\, j(\phi) \oplus  
\{ X \cdot \phi \, : \, X \in \R^6 \} .
\ee
In particular, $\js$ preserves the subspaces $\{ X \cdot \phi \, : \, X \in \R^6 \}
\subset \Delta$, and the formula
\bdm
\J(X) \cdot \phi \ := \ \js(X \cdot \phi)
\edm
defines an orthogonal complex structure $\J$ on $\R^6$ 
that depends on $\phi$. 
Moreover, the spinor determines a $3$-form by means of
\bdm
\psis(X,Y,Z) \ := \ -(X \cdot Y \cdot Z \cdot \phi \, , \, \phi)
\edm
where the brackets indicate the inner product on $\Delta$.
The pair $(\J , \psis)$ is an $\SU(3)$-structure on $\R^6$, and 
any such arises in this fashion from some real spinor. 
In certain cases this is an established 
 construction: a nearly K\"ahler structure 
may be recovered from the Riemannian Killing spinor \cite{Grunewald90}, 
for instance. All this can be summarised in the known fact
that $\SU(3)$-structures on $\R^6$ correspond one-to-one with
real spinors of length one $(\bmod \Z_2)$,
\bdm
\SO(6)/\SU(3)  
 \ \iso\ \P(\Delta) \ = \ \RP^7 .
\edm
\begin{exa}
Consider the spinor $\phi = (0,0,0,0,0,0,0,1) \in \Delta = \R^8$. With the basis chosen on p. \pageref{SU3basis}, then, $\J$ and $\psis$ read
\bdm
\J e_1 \, = \, -e_{2}, \quad \J e_3 \, = \, e_{4}, \quad 
\J e_5 \, = \, e_{6},\ 
\psis \ = \  e_{135} \, - \, e_{146} \, + \, e_{236} \, + e_{245},  
\edm
\end{exa}
where $e_{135}=e_1\w e_3\w e_5$ \&c. Throughout this article $e_i$ indicate  
tangent vectors and one-forms indifferently.
\smallbreak

Below we summarise formulas expressing the action of $\J$ and
$\psis$, whose proof is an easy exercise in local coordinates and so omitted.
\begin{lem}\label{lem:algrechnungendim6}
For any unit spinors $\phi, \phib$ and any vector $X \in \R^6$
\bdm
\psis \cdot \phi \, = \, -4 \cdot \phi , \quad
\psis \cdot \js(\phi) \, = \,  4 \cdot \js(\phi) , \quad
\psis \cdot \phib \, = \, 0  \;\; \text{if } \ \phib \perp \phi , \js(\phi) ,
\edm
\bdm
(X \haken \psis) \cdot \phi \, = \, 2 \, X \cdot \phi, \quad
\J \phi \, = \, 3 \, \js(\phi) , \quad   \J\left( \jphis\right) \, = \, - 3 \, \phi .
\edm
\end{lem}
\bigbreak
In dimension $7$ the space $\Delta$ does not carry an invariant 
complex structure akin to $\js$. 
However, we still have a decomposition. If we take a 
non-trivial real spinor $0 \not= \phi \in \Delta$, we may split 
\be\label{G2-dec}
\Delta \ = \ \R \phi \, \oplus \ 
\{ X \cdot \phi \, : \, X \in \R^7 \}  ,
\ee
and we can still define a $3$-form
\bdm
\f(X,Y,Z)\ :=\ (X\cdot Y\cdot Z \cdot \phig,\phig).
\edm
It turns out that $\f$ is stable (its $\GL(7)$-orbit is open), and its 
isotropy group inside 
$\GL(7,\R)$  is isomorphic to the exceptional 
Lie group $\G\subset \SO(7)$. Thus we recover the renowned
fact that
there is a one-to-one correspondence between positive stable $3$-forms 
$\Psi \in \Lambda^3\R^7$ of fixed length  and real lines in $\Delta$:
\bdm
\SO(7)/\mathrm{G}_2\ \iso\ \P(\Delta) \ = \ \RP^7 \ .
\edm
In analogy to Lemma \ref{lem:algrechnungendim6}, here are formulas to be used in the sequel.
\begin{lem}\label{lem:algdim7}
Let $\f$ be a stable three-form on $\R^7$ inducing the spinor $\phi$, and suppose $\phib$ is a unit spinor orthogonal to $\phi$. Then
\bdm
\f \cdot \phi \, = \, 7 \phi , \quad 
\f \cdot \phi^* \, = \, - \phib , 
\quad (X \haken \f) \cdot \phi=-3X\cdot \phi. 
\edm
\end{lem}
\begin{NB}
The existence of the unit spinor $\phi$ on $M^6$ is a general fact. 
Any $8$-dimensional real vector bundle over a $6$-manifold admits a 
unit section, see e.\,g. \cite[Ch. 9, Thm. 1.2]{Husemoller}. 
Consequently, an oriented Riemannian 
 $6$-manifold admits a spin structure if and only if it admits a reduction from
$\Spin(6)\cong \SU(4)$ to $\SU(3)$. 
The argument also applies to  $\Spin(7)$- and $\G$-structures, and was 
practised extensively in  \cite[Prop. 3.2]{FKMS}.
\end{NB}
%
\smallskip
The power of the approach presented in this paper is already manifest at this stage. Consider a $7$-dimensional Euclidean space $\bar U$ equipped with a 
$\G$-structure $\Psi \in \Lambda^3\bar U$. The latter induces an $\SU(3)$-structure 
on any codimension-one subspace $U$, which may be defined in two ways. 
One can restrict $\Psi$ to $U$, so that the inner 
product $V \haken \Psi$ with a normal vector $V$ defines a complex 
structure on $U$. 
But it is much simpler to remark that both structures, on $\bar U$ and $U$,  
correspond to the same choice of the real spinor $\phi \in \Delta$.  
%
%
\section{Special Hermitian geometry}\label{sec:tysu3}
%
%
The premises now in place, an $\SU(3)$-manifold will be a 
Riemannian spin manifold $(M^6,g,\phis)$ 
equipped with a global spinor $\phis$ of length one. We always denote 
its spin bundle by $\Sigma$ and the corresponding Levi-Civita connection by 
$\nabla$. The induced $\SU(3)$-structure is determined by the $3$-form 
$\psis$, while the $2$-form   $\omega(.\,,.)=g(.\,,J.)$ defines the 
underlying $\U(3)$-structure. From now onwards we will drop the symbol 
for the Clifford product, so $X\cdot \phi$ will simply read $X\phi$. 
\begin{dfn}
By decomposition (\ref{eq:splittingsu3}) there
exist a unique one-form $\eta\in T^*M^6$ and a unique section 
$S\in\End(TM^6)$ such that
\be\label{spineq.SU3}
\nabla_X \phis=\eta(X)  \jphis+S(X)\phi.
\ee
We call $S$ the \emph{intrinsic endomorphism} 
and  $\eta$ the \emph{intrinsic $1$-form} of the $\SU(3)$-manifold 
$(M^6,g,\phis)$; this terminology will be fully justified by Proposition \ref{prop:intrtorsu}. 
\end{dfn}
Recall that the geometric features of $M^6$ are 
captured \cite{SMS:redbook} (see also \cite{Ilka:Srni} and \cite{Friedrich03b}) 
by the intrinsic torsion $\Gamma$ which,  under 
$$\Lambda^2T^*M^6 \iso \so(6)=\su(3)\oplus\su(3)^\perp,$$
becomes a one-form with values in $\su(3)^\perp$.
For instance, nearly K\"ahler manifolds are those almost Hermitian manifolds 
for which $\Gamma$, identified with $\nabla\omega\in \Lambda^1\otimes \su(3)^\perp$, is skew:  $\nabla_X\omega(X,Y)=0,\ \forall\,X,Y$.
The aim is to recover the various $\SU(3)$-classes (complex, symplectic, lcK\ldots) essentially by reinterpreting the intrinsic torsion using $S$ and $\eta$. 
Besides $\psis$, we have 
a second, so-to-speak fundamental $3$-form
\bdm
\psims(X,Y,Z):=\psis(JX,JY,JZ)=-\psis(JX,Y,Z)=-(XYZ\phis,\jphis),
\edm
which gives the imaginary part of a  $\J$-holomorphic complex $3$-form 
(the real part being $\psis$).
As a first result, we prove that the intrinsic torsion can be expressed 
through $S$ and $\psims$, while $\eta$ is related to $\nabla\psims$---thus
generalizing the well-known definition of nearly K\"ahler manifolds cited above.
\begin{lem}\label{lem:nablaompsip} 
The intrinsic endomorphism $S$ and the intrinsic $1$-form $\eta$ are 
related to $\nabla\omega$ and $\nabla\psims$ through ($X,Y,Z$ any vector fields)
\bdm
(\nabla_X\omega)(Y,Z)=2\,\psims(S(X),Y,Z) ~~~ \mbox{ and } ~~~ 
8\, \eta(X)=-(\nabla_X\psims)(\psis).
\edm
\end{lem}
\begin{proof}
We immediately find $\eta=(\nabla\phi,\jphis)$. Since $\js$ can be thought 
of as the volume form, it is parallel under $\nabla$ and we conclude
\bdm
\nabla_X(\jphis)=\js\nabla_X\phis=\js S(X)\phis+\js\eta(X)\jphis
=-S(X)\jphis-\eta(X)\phis.
\edm
With $\omega(X,Y)=-(X\phis,Y\jphis)$ we get
\bea[*]
-\nabla_X\omega(Y,Z)&=&X (Y\phis,Z\jphis) -(\nabla_XY\phis,Z\jphis)
-(Y\phis,\nabla_XZ\jphis)\\
& =& (Y\nabla_X\phis,Z\jphis)+(Y\phis,Z\nabla_X\jphis)
 = (YS(X)\phis,Z\jphis)-(Y\phis,ZS(X)\jphis)
\\
%
&=&-2\psims(S(X),Y,Z) .
\eea[*]
Similarly, we compute
\bea[*]
\nabla_X(\psims)(\psis)&=&-X(\psis\phis,\jphis)+(\nabla_X\psis\phis,\jphis) \\ 
& =& -(\psis S(X)\phis,\jphis)+(\psis\phis,S(X)\jphis)
-\eta(X)(\psis\jphis,\jphis)+\eta(X)(\psis\phis,\phis)\\
&=&2\eta(X)(\psis\phis,\phis)=-8\eta(X) .
\eea[*]
This finishes the proof.
\end{proof}
To understand the role of the pair $(S, \eta)$ better we shall employ 
the $\SU(3)$-connection 
\be\label{nabla^n}
\nabla^n_XY=\nabla_XY-\Gamma(X)(Y),
\ee
given by the Levi-Civita connection $\nabla$ minus  the intrinsic 
torsion, see \cite{SMS:redbook, Friedrich03b}. 
We shall always use only one symbol for covariant derivatives on the 
tangent bundle and their liftings to the spinor bundle $\Sigma$, whence
for any spinor $\phib$
\bdm
\nabla^n_X\phib=\nabla_X\phib-\frac{1}{2}\Gamma(X)\phib.
\edm
\begin{prop}\label{prop:intrtorsu}
The intrinsic torsion of the $\SU(3)$-structure $(M^6,g,\phis)$ is given by 
\bdm
\Gamma=S\hook\psis-\frac{2}{3}\eta\otimes\omega
\edm
where $S\hook\psis(X,Y,Z):=\psis(S(X),Y,Z)$.
\end{prop}
\begin{proof}
The spinor $\phis$ is parallel for $\nabla^n$, as 
$\textrm{Stab}(\phis)= \SU(3)$, so 
$ \nabla_X\phis=\frac{1}{2}\Gamma(X)\phis$. 
By  Lemma \ref{lem:algrechnungendim6} we know that $\omega\phis=-3\jphis$, 
hence
\bdm
\nabla_X\phis=S(X)\phis+\eta(X)\jphis=\frac{1}{2}(S(X)\hook\psis)
\phis-\frac{1}{3}\eta(X)\omega\phis.
\edm
Since $(X\hook\psis)(Y,\J Z)=(X\hook\psis)(\J Y,Z)$ we see that 
$X\hook\psis\in\su(3)^\perp$, and as $\omega\in\su(3)^\perp$ the 
$1$-form $ 
S\hook \psis-\frac{2}{3}\eta\otimes \omega$  is the intrinsic torsion of 
the spin connection.  
\end{proof}
\begin{nota}
The original approach to the classification of
$\U(3)$-manifolds in  \cite{GH80} was by the covariant
derivative of the K\"ahler form. In analogy to their result,
one calls the seven `basic' irreducible 
modules of an $\SU(3)$-manifold the \emph{Gray-Hervella classes}. 
Throughout this paper 
they will be indicated $\WS_1^+,\; \WS_1^-,\;  \WS_2^+,\;  \WS_2^-,\;  
\WS_3,\;  \WS_4,\;  \WS_5$; 
for simplicity we often will write $\WS_j, \WS_{\bar j}$ for $\WS_j^+, \WS_j^-$ respectively, and shorten $\WS_1^+\oplus\WS_2^-\oplus\WS_4$ to 
$\WS_{1\bar24}$, \&c.
In \cite{CS02} the Gray-Hervella classes of  $\SU(3)$-manifolds
were derived in terms of the components of the intrinsic torsion, while
their identification with the covariant derivatives of the K\"ahler form 
and of the complex volume form may be found in \cite{Martin-Cabrera05}.
\end{nota}
The following result links the intrinsic
endomorphism $S$ and the intrinsic $1$-form $\eta$ (and thus the spinorial 
field equation (\ref{spineq.SU3}))
directly to the Gray-Hervella classes $\WS_i$.
\begin{lem}\label{lem:classsu3}
The basic classes of an $\SU(3)$-structure $(M^6,g,\phis)$ are 
determined as follows, where $\lambda, \mu\in\R$:
\begin{center}
\begin{tabular}{|c|c|c|}
\hline
class&description&dimension\\
\hline
 \hline
$\WS_1$& $S=\lambda\,\J$, $\eta=0$&$1$\\
\hline
$\WS_{\bar1}$& $S=\mu\, \Id$, $\eta=0$&$1$\\
\hline
$\WS_2$& $S\in\su(3)$, $\eta=0$&$8$\\
\hline
$\WS_{\bar2}$& $S\in\{A\in S_0^2T^*M|A\J=\J A\}$, $\eta=0$&$8$\\
\hline
$\WS_3$& $S\in\{A\in S_0^2T^*M|A\J=-\J A\}$, $\eta=0$&$12$\\
\hline
$\WS_4$& $S\in\{A\in\Lambda^2(\R^6)|A\J=-\J A\}$, $\eta=0$&$6$\\
\hline
$\WS_5$& $S=0$, $\eta\neq0$&$6$\\
\hline
\end{tabular}
\end{center}
\end{lem}
In particular, the class is $\WS_{\bar1\bar23}$ if and only if 
$S$ is symmetric and $\eta$ vanishes, recovering a result of \cite{CS06}.
\subsection{Spinorial characterisation}
The description of $\SU(3)$-structures in terms of $\phis$ is the main 
result of this section. 
To start with, we discuss geometric quantities that pertain the 
$\SU(3)$-structure and 
how they correspond to $\phis$. Denote by $D$ the Riemannian Dirac operator.
\begin{lem}\label{lem:omsu3}
The $\WS_4$ component of the intrinsic torsion of an $\SU(3)$-manifold  is 
determined by
\bdm
\delta\omega(X)=2[(D\phis,X\jphis)-\eta(X)],
\edm
and in particular $\delta\omega=0$ is equivalent to $(D\phis,X\jphis)=\eta(X)$.
The Lee form is given by 
\bdm
\theta(X)=\delta\omega\circ J(X)=2(D\phis,X\phis)-2\eta\circ J(X).
\edm
\end{lem}
\begin{proof}
We have
\bdm
(\nabla_X\omega)(Y,Z)=(ZY\nabla_X\phis,\jphis)+(ZY\phis,\nabla_X\jphis)
=-2(YZ\nabla_X\phis,\jphis)-2g(Y,Z)\eta(X),
\edm
leading to
\bea[*]
\delta\omega(X)&=&-\sum_i(\nabla_{e_i}\omega)(e_i,X)
=\sum_i(\nabla_{e_i}\omega)(X,e_i)\\
&=&-2\sum_i((Xe_i\nabla_{e_i}\phis,\jphis)-g(X,e_i)\eta(e_i))\\
&=&-2(XD\phis,\jphis)-2\eta(X)=2(D\phis,X\jphis)-2\eta(X).
\eea[*]
\vspace{-1.3cm}\\
\end{proof}
\vspace{4mm}

We consider the space of all possible types 
$T^*M^6\otimes\phis^{\perp}\ni\nabla\phis$, where 
$\phis^{\perp}=\R\jphis\oplus\{X\phis~|~X\in TM^6\}$ 
is the orthogonal complement of 
$\phis$. The Clifford multiplication restricts then to a map 
\bdm
  \cliffmult: \ T^*M^6\otimes\phis^{\perp}\ra\Sigmas.
\edm
Let $\pi:\Spin(6)\ra\SO(6)$ be the usual projection. For any $h\in\Spin(6)$ we have 
\bdm
\cliffmult(\pi(h)\eta\otimes h\phib)=h\eta h^{-1}h\phib=h\cliffmult(\eta\otimes\phib)
\edm
and $\cliffmult$ is $\Spin(6)$-equivariant and thus $\SU(3)$-equivariant. Comparing the 
dimensions of the modules appearing in (\ref{eq:splittingsu3}) and the ones of 
Lemma \ref{lem:classsu3} we see that $\WS_{2\bar23}\subset \Ker(\cliffmult)$, and using 
\bdm
D\phis=6\lambda\jphis \mbox{ for } S=\lambda \J\quad \mbox{ and } \quad 
D\phis=-6\mu\phis \mbox{ for } S=\mu \Id
\edm
we find correspondences 
\bdm 
\WS_1\ra \R\jphis\; \mbox{ and }\; \WS_{\bar1}\ra \R\phis,
\edm
together with $(D\phis,\jphis)=6\lambda$ and $(D\phis,\phis)=-6\mu$.
\smallbreak 

Let us look at $\WS_{45}$ closer: recall that $\{\J e_i\phis,\phis,\jphis\},\ i=1,\ldots, 6$ is a basis 
of $\Sigmas$ for some local orthonormal frame $e_i$, hence
\bea[*]
D\phis&=&\sum_{i=1}^6(D\phis,\J e_i\phis)\J e_i\phis+(D\phis,\phis)\phis+(D\phis,\jphis) \jphis.
\eea[*]
With Lemma \ref{lem:omsu3} we conclude that
\bea[*]
D\phis&=&\sum_{i=1}^6[\tfrac{1}{2}\delta\omega(e_i)+\eta(e_i)]e_i\jphis
+6\lambda\jphis-6\mu\phis=(\tfrac{1}{2}\delta\omega+\eta)\jphis+6\lambda\jphis-6\mu\phis.
\eea[*]
Therefore, as image of $m$, the component $\R^6$ of $\Sigma$ 
is determined by $\delta\omega+2\eta$. This line of thought immediately 
proves
\begin{thm}\label{thm:Dphisu3}
A $6$-dimensional Riemannian spin manifold $(M,g)$ carries a unit spinor 
$\phis$ lying in the kernel of the Dirac operator 
\bdm
D\phis=0
\edm
if and only if it admits an $\SU(3)$-structure of class $\WS_{2\bar2345}$ 
with the restriction $\delta\omega=-2\eta$. \\
The `complementary' torsion components 
$\WS_1$ and $\WS_{\bar1}$ are determined by the scalars
\bdm
\lambda\ =\ \frac 16(D\phis,\jphis)\ =\ -\frac 16 \tr(\J S) 
\quad \mbox{ and } \quad
\mu \ =\ -\frac16 (D\phis,\phis)\ =\ \frac 16\tr(S).
\edm
\end{thm}
%
%
One cannot but notice that harmonic spinors can exist on manifolds 
whose class is the opposite to that of nearly K\"ahler manifolds. 
The consequences of this observation remain -- at this stage -- to be seen, and 
 will be addressed in forthcoming work.\smallbreak 

The linear combination $\WS_4+2\WS_5$ vanishing in the theorem also shows up 
(up to a choice of volume) in \cite{Cardosoetal03} and plays a 
role in supersymmetric compactifications of heterotic string theory.
\begin{exa}
Consider the Lie algebra $\g=\textrm{span}\{e_1,\ldots,e_6\}$ 
with structure equations
$$
\vec{d}=(e_{34}+2e_{35},e_{45},0,0,0,e_{51}+e_{23})
$$
in terms of the Chevalley-Eilenberg differential 
$d\beta(e_i,e_j)=\beta[e_j,e_i], \forall\beta\in\g$.
Since the structure constants are rational the corresponding $1$-connected Lie 
group $G$ contains a co-compact lattice $\Lambda$. We consider the spin structure 
on $M^6=G/\Lambda$ 
determined by choosing $\phis=(1,1,0,0,0,0,0,0)$. This gives us
\bdm
S=- \frac{1}{4}\left[\begin{smallmatrix}
   0&0&&&&\\
   -2&0&&&&\\
  &&0&1&-1&0\\
  &&1&0&0&-1\\
  &&-1&0&0&-1\\
  &&0&1&1&0
  \end{smallmatrix}\right], \q
\eta=- \frac{1}{2} e_1 
\edm
and it is not hard to see that $D\phi=0$. The structure is of class
$\WS_{2\bar2345}$, and the presence of component nr.~$5$ is 
reflected in the non-vanishing $\eta$. 
\end{exa}
\begin{nota}\label{eq:decompS}
Recalling Lemma \ref{lem:classsu3} we decompose the intrinsic endomorphism into
\bdm
S=\lambda\,\J+\mu\,\Id+S_2+S_{34}
\edm
where $\J$ commutes with $S_2$, anti-commutes with $S_{34}$, and both $S_2$ and $\J S_2$ are 
trace-free.
\end{nota}
\bigbreak

The next results discusses the Nijenhuis tensor $N_J(X,Y)=8\Re[X^{1,0},Y^{1,0}]^{0,1}$, whose vanishing tells that $M$ is a complex manifold. The customary trick in a Riemannian setting is to view it as a three-tensor $N(X,Y,Z)=g(N_J(X,Y),Z)$ by contracting with the metric.
\begin{lem}\label{lem:Niejn}
The $\WS_{1\bar12\bar2}$ component is controlled by the Nijenhuis tensor
\bdm
N(X,Y,Z)\ =\ -2[\psims((\J S+S\J)X,Y,Z)-\psims((\J S+S \J)Y,X,Z)].
\edm
Therefore if the class is $\WS_{1\bar1345}$, the Nijenhuis tensor reads 
\bdm
N(X,Y,Z)\ =\ 8[\lambda\psims(X,Y,Z)-\mu \psis(X,Y,Z)].
\edm
\end{lem}
\begin{proof}
 We have $g((\nabla_X\J)Y,Z)=-(\nabla_X\omega)(Y,Z)$, and from Lemma \ref{lem:nablaompsip} 
\begin{align*}
N(X,Y,Z)&= -(\nabla_X\omega)(\J Y,Z)+(\nabla_Y\omega)(\J X,Z)
-(\nabla_{\J X}\omega)(Y,Z)+(\nabla_{\J Y}\omega)(X,Z)\\
&=2[-\psims(SX,\J Y,Z)+\psims(SY,\J X,Z)-\psims(S\J X,Y,Z)+\psims(S\J Y,X,Z)]\\
&=2[-\psims((\J S+S\J )X,Y,Z)+\psims((\J S+S\J )Y,X,Z)].
\end{align*}
Futhermore for $S=\lambda\J+\mu\Id+S_{34}$ we have
\bdm
\J S+S\J=\J(S_{34}+\lambda \J+\mu \Id)+(S_{34}+\lambda \J+\mu \Id)\J
=-2\lambda \Id+2\mu \J,
\edm
as claimed.
\end{proof}
Eventually, $\WS_{1\bar134}$ depends on $d\omega$ in the following way: 
\begin{lem}\label{lem:dom}
Retaining Notation \ref{eq:decompS} we have
\bdm
d\omega(X,Y,Z)=6\lambda\psis(X,Y,Z)+6\mu\psims(X,Y,Z)
+2\cyclic{XYZ}\psims(S_{34}(X),Y,Z).
\edm
\end{lem}
\begin{proof}
We have $d\omega(X,Y,Z)=\cyclic{XYZ}(\nabla_X\omega)(Y,Z)$, and  the fact that 
$\cyclic{XYZ}\psims(S_2(X),Y,Z)$ vanishes corresponds to $d\omega=0$ 
in $\WS_{2\bar25}$.
\end{proof}
To attain additional equations in terms of $\phis$, thus completing the picture, 
we need one last technicality.
\begin{lem}\label{A1} 
The intrinsic tensors $(S,\eta)$ of a Riemannian spin manifold $(M^6,g,\phis)$ satisfy  
the following properties:
$$\ba{rcl}
 S, \J \mbox{ commute}  &\Longleftrightarrow& (\J Y\nabla_X\phis,\phis)
=-(Y\nabla_{\J X}\phis,\phis),\\
 S, \J \mbox{ anti-commute}  &\Longleftrightarrow& (\J Y\nabla_X\phis,\phis)
=(Y\nabla_{\J X}\phis,\phis),\\
 S \mbox{ is symmetric } &\Longleftrightarrow& (X\nabla_Y\phis,\phis)
=(Y\nabla_X\phis,\phis),\\
 S \mbox{ is skew-symmetric }  &\Longleftrightarrow&  (X\nabla_Y\phis,\phis)
=-(Y\nabla_X\phis,\phis).
\ea$$
\end{lem}
\begin{proof}
As $(\J S(X)\phis,Y\phis)=(S\J(X)\phis,Y\phis)$ if and only if
$(\J Y\nabla_X\phis,\phis)=-(Y\nabla_{\J(X)}\phis,\phis)$,
the first two equivalences are clear. 

Since both $\phis,\jphis$ are orthogonal to $Y\phis$, for any $Y\in TM^6$, we obtain
\bdm
g(S(X),Y)=(\nabla_X\phis,Y\phis) \  \mbox{ and } \  g(X,S(Y))=(\nabla_Y\phis,X\phis)
\edm
and hence the remaining formulas.
\end{proof}
\begin{thm}\label{thm:class}
The classification of $\SU(3)$-structures in terms of the defining 
spinor $\phis$ is contained 
in Table \ref{table.SU3}, where 
$$\eta(X):=(\nabla_X\phis,\jphis)$$ 
and $\lambda=\frac{1}{6}(D\phis,\jphis),\; \mu=-\frac{1}{6}(D\phis,\phis)$ (as of Theorem $\ref{thm:Dphisu3}$).

\begin{table}
\caption{Correspondence of $\SU(3)$-structures and spinorial field equations
(see Theorem  \ref{thm:class}).}\label{table.SU3}
\begin{tabular}{|c|l|}
\hline
class & spinorial equations\\
\hline
 \hline
$\WS_1$& $\nabla_X\phis=\lambda X\jphis$ for $\lambda\in\R$\\
\hline
$\WS_{\bar1}$& $\nabla_X\phis=\mu X\phis$ \ \ \ for   $\mu\in\R$\\
\hline
$\WS_2$& $(\J Y\nabla_X\phis,\phis)=-(Y\nabla_{\J X}\phis,\phis)$, \quad
$(Y\nabla_X\phis,\jphis)=+(X\nabla_Y\phis,\jphis)$, $\ \lambda=\eta=0$\\
\hline
$\WS_{\bar2}$& $(\J Y\nabla_X\phis,\phis)=-(Y\nabla_{\J X}\phis,\phis)$, \quad
$(Y\nabla_X\phis,\jphis)=-(X\nabla_Y\phis,\jphis)$, $\ \mu=\eta=0$\\
\hline
$\WS_{3}$& $(\J Y\nabla_X\phis,\phis)=+(Y\nabla_{\J X}\phis,\phis)$, \quad
$(Y\nabla_X\phis,\jphis)=+(X\nabla_Y\phis,\jphis)$, \  $\eta=0$  \\
\hline
$\WS_{4}$& $(\J Y\nabla_X\phis,\phis)=+(Y\nabla_{\J X}\phis,\phis)$, \quad
$(Y\nabla_X\phis,\jphis)=-(X\nabla_Y\phis,\jphis)$, \ $\eta=0$ \\
\hline
$\WS_5$& $\nabla_X\phis=(\nabla_X\phis,\jphis)\jphis$\\
\hline\hline
$\WS_{1\bar1}$& $\nabla_X\phis=\lambda X\jphis+\mu X\phis$\\
\hline
$\WS_{2\bar2}$& $(\J Y\nabla_X\phis,\phis)=-(Y\nabla_{\J X}\phis,\phis)$, 
$\lambda=\mu=0$ and $\eta=0$ 
\\
\hline
$\WS_{2\bar25}$& $(\J Y\nabla_X\phis,\phis)=-(Y\nabla_{\J X}\phis,\phis)$ and 
$\lambda=\mu=0$\\
\hline
$\WS_{1\bar12\bar2}$& $(\J Y\nabla_X\phis,\phis)=-(Y\nabla_{\J X}\phis,\phis)$ and 
$\eta=0$\\
\hline
$\WS_{1\bar12\bar25}$& $(\J Y\nabla_X\phis,\phis)=-(Y\nabla_{\J X}\phis,\phis)$\\
\hline
$\WS_{2\bar23}$& $D\phis=0$ and $\eta=0$\\
\hline
$\WS_{1\bar12\bar23}$& $(D\phis,X\phis)=0$ and $\eta=0$\\
\hline
$\WS_{1\bar12\bar234}$& $(\nabla_X\phis,\jphis)=0$\\
\hline
$\WS_{2\bar235}$& $(D\phis,X\jphis)=\eta(X)$ and $\lambda=\mu=0$ \\
\hline
$\WS_{1\bar12\bar235}$& $(D\phis,X\jphis)=\eta(X)$\\
\hline
$\WS_{34}$& $(\J Y\nabla_X\phis,\phis)=(Y\nabla_{\J X}\phis,\phis)$ and $\eta=0$\\
\hline
$\WS_{345}$& $(\J Y\nabla_X\phis,\phis)=(Y\nabla_{\J X}\phis,\phis)$\\
\hline
$\WS_{2\bar2345}$& $\lambda=\mu=0$\\
\hline
$\WS_{\bar1\bar23}$& $(X\nabla_Y\phis,\phis)=(Y\nabla_X\phis,\phis)$ and $\eta=0$\\
\hline
\end{tabular}
\end{table}
\end{thm}
\begin{proof}
We first prove that $\lambda$ and $\mu$ in $\WS_1$ and $\WS_{\bar1}$ are 
constant. In 
$\WS_1$ we have $S=\lambda \J$ and thus 
$\nabla_X(\phis+\jphis)=-\lambda X(\phis+\jphis)$. 
Since a nearly K\"ahler structure (type $\WS_{1\bar15}$) is given by a Killing 
spinor
 \cite{Grunewald90}, the function $\lambda$ must be constant. In the 
 case $\WS_{\bar1}$ the 
spinors $\phis$ and $\jphis$ themselves are Killing spinors with Killing constants 
$\mu$, $-\mu$. \vspace{0.2cm}\\
We combine the results of Lemma \ref{A1} as follows.
By Lemma \ref{lem:classsu3}, a structure is of type $\WS_2$ if $S$ is skew-symmetric, 
it commutes with $\J$, 
and the trace of $\J S$ and $\eta$ vanish. 
The first statement of Lemma \ref{A1} gives us the condition for
$S$ and  $\J$ to commute, and the last states that skew-symmetry is equivalent
to  
$(X\nabla_Y\phis,\phis)=-(Y\nabla_X\phis,\phis)$, under which condition 
the equation $(\J Y\nabla_X\phis,\phis)=-(Y\nabla_{\J X}\phis,\phis)$ 
is equivalent to 
\bdm
(Y\nabla_X\phis,\jphis)=
(X\nabla_Y\phis,\jphis).
\edm
The other classes can be calculated similarly, making extensive use of 
Lemmas \ref{lem:omsu3}, \ref{A1}.
\end{proof}
It makes little sense to compute all possible combinations (in principle, $2^7$), 
so we listed only those of some interest. Others can be inferred by 
arguments of the following sort. 
Suppose we want to show that class $\WS_{124}$ has 
$(X\nabla_Y\phis,\phis)=-(Y\nabla_X\phis,\phis)$ and $\eta=0$ as 
defining equations. 
>From Lemma \ref{lem:classsu3} we know $\WS_{124}$ is governed by the 
skew-symmetry of $S$, and at the same time $\eta$ controls $\WS_5$, whence 
the claim is straighforward. 
Another example: assume we want to show that 
\bdm
3\lambda\psis(X,Y,Z)+3\mu\psims(X,Y,Z) +
\cyclic{XYZ}(YZ\nabla_X\phis,\jphis)+\cyclic{XYZ}\eta(X)g(Y,Z)
\ =\ 0
\edm
is an alternative description of  $\WS_{1\bar12\bar25}$. 
>From Lemma \ref{lem:dom} we know that 
$d\omega=6\lambda\psis+6\mu\psims$ defines that class, so we conclude by 
using $d\omega(X,Y,Z)=\cyclic{XYZ}(\nabla_X\omega)(Y,Z)$ and the first equality  
in the proof of Lemma \ref{lem:omsu3}.

%
%
%
\begin{NBs}\label{NB:W1}~\\
(i) The proof above shows that the real Killing spinors of an $\SU(3)$-structure 
of  class $\WS_{1\bar1}$ (with Killing constants $\pm|\lambda|$) necessarily 
have the form
$\phis\pm\jphis$ in case $\WS_1$, and
$\phis,\, \jphis$   in case  $\WS_{\bar1}$.
%
Now notice that a rotation of $\phis$ to $\phis\cos\alpha+\jphis\sin\alpha$, for some 
function $\alpha$, affects  
the intrinsic tensors as follows:
\bdm
S \rightsquigarrow S\cos(2\alpha)+\J\circ S\sin(2\alpha),\qquad 
\eta\rightsquigarrow \eta+ d\alpha
\edm 
The $\WS_5$ component varies, and $\WS_i^\pm, i=1,2$ change, too \cite{CS02}.
\smallbreak

In class $\WS_{\bar123}$ we have the constraint $D\phis=f\phis$, 
so $\phis$ is an eigenspinor with eigenfunction $f$. (One can alter  
$\phis$ so to have it in $\WS_{1\bar12\bar23}$.)
Therefore, if we are after a Killing spinor (class $\WS_{1\bar1}$), the 
eigenfunction 
$f$ necessarily determines the fifth component $\eta=-d\alpha$. 
In Section \ref{sec:spincones} we will treat cases where $f=h$ is a constant map.

\smallbreak

(ii) It is fairly evident (cf. \cite{CS02}) that the effect of modifying  
$S\rightsquigarrow JS$ is to exchange  
$\WS_j^+$ and $\WS_j^-, \ j=1,2$, whilst the other components 
remain untouched. As such it corresponds to a $\pi/2$-rotation in the fibres of the 
natural circle bundle $\RP^7\lto \CP^3$. 
\end{NBs}
\bigbreak
\begin{exa}
The twistor spaces $M^6=\C\P^3, \ \U(3)/\U(1)^3$ of the 
self-dual Einstein manifolds $S^4$ and $\C\P^2$ are very interesting from
the spinorial point of view. As is well known, both carry a one-parameter 
family of metrics $g_t$ compatible with two almost complex structures
$\Omega^\mathrm{K}, \Omega^\mathrm{nK}$, in such a way that 
in a suitable, but pretty standard normalisation 
$(M^6,g_{1/2},\Omega^\mathrm{nK})$
is nearly K\"ahler and $(M^6,g_{1},\Omega^\mathrm{K})$ is K\"ahler \cite{ES85}.
The two  almost complex structures differ by an orientation flip on the 
two-dimensional fibres. Here is a short and uniform description
of both instances. 
We choose the 
spin representation used in \cite[Sect. 5.4]{BFGK91}, 
whereby the Riemannian scalar curvature of $g_t$ is 
\bdm
\Scal_t\ =\ 2c (6-t+1/t) 
\edm
where $c$ is a constant (equal to $1$ for $\C\P^3$ and $c=2$ for 
$\U(3)/\U(1)^3$, due to an irrelevant yet nasty factor of $2$ 
in standard normalisations).
Using an appropriate orthonormal frame the orthogonal almost 
complex structures read 
\bdm
\Omega^\mathrm{K}\ =\ e_{12} - e_{34}- e_{56},\quad
\Omega^\mathrm{nK}\ =\ e_{12}- e_{34}+ e_{56}.
\edm
There exist two linearly independent and isotropy-invariant
 real spinors $\phi_\eps$  in $\Delta\ (\eps=\pm 1)$, which define
global spinor fields on the two spaces. 
On can prove directly that the $\phi_\eps$ induce the same $J_\phi$, 
corresponding to $\Omega^\mathrm{nK}$, and also the $3$-forms
\bdm
\psi_\eps \ :=\ \psi_{\phi_\eps}\ =\ \eps (e_{135}+e_{146}-e_{236} + e_{245})\ =:\ \eps \Psi.
\edm
When $t=1/2$, $\phi_\eps$ are known to be Killing spinors. 
For a generic $t\not =0$ let us define the
symmetric endomorphisms $S_\eps:\ TM^6\ra TM^6$
\bdm
S_\eps\ =\ \eps\, \sqrt{c}\cdot \diag \bigg(\frac{\sqrt{t}}{2},\frac{\sqrt{t}}{2},
\frac{\sqrt{t}}{2},\frac{\sqrt{t}}{2},
\frac{1-t}{2\sqrt{t}}, \frac{1-t}{2\sqrt{t}} \bigg).
\edm
An explicit calculation shows that  $\phi_\eps$ solve
\bdm
\nabla_X\phi_\eps\ =\  S_\eps(X) \, \phi_\eps,
\edm
making them generalised Killing spinors.
In particular, $S_\eps$ are the intrinsic endomorphisms and 
$\eta=0$; observe that $S_\eps$ commute with 
$\Omega^\mathrm{nK}$ due to their block structure. 
By Lemma \ref{lem:classsu3} the
$\SU(3)$-structure defined by $\phi_\eps$ is 
therefore of class $\chi_{\bar{1}\bar{2}}$ for $t\neq 1/2$, and reduces to 
class $\chi_{\bar{1}}$ when $t=1/2$. \smallbreak

The spinors $\phi_\eps$ are eigenspinors of
the Riemannian Dirac operator $D$ with eigenvalues
$6\mu = \tr S_\eps = \eps\sqrt{c}\,\frac{t+1}{\sqrt t}$; 
they coincide, as they should, with the limiting values of
Friedrich's general estimate \cite{Fr80} when $t=1/2$,   and 
Kirchberg's estimate for K\"ahler manifolds \cite{Kirchberg86} for $t=1$.
\smallbreak

A further  routine calculation shows that 
\bdm
\nabla_{e_i} \Omega^\mathrm{nK} \, =\, 
\begin{cases}
-\sqrt{ct}Je_i\lrcorner \Psi &\quad 1\leq i\leq 4 \\[3mm]
-\dfrac{\sqrt{c}\,(1-t)}{\sqrt{t}}Je_i \lrcorner \Psi &\quad i=5,6.
\end{cases}
\edm
%
Hence, we conclude that $\nabla_X \Omega^\mathrm{nK} = -2JS_\eps(X)\lrcorner\psi_\eps $
holds, as it should by Lemma \ref{lem:nablaompsip}.\smallbreak

Let us finish with a comment on the K\"ahler structures ($t=1$). 
Kirchberg's equality is attained in odd complex 
dimensions by a pair of so-called K\"ahlerian Killing spinors, 
 basically $\phi_1, \phi_{-1}$ \cite{Kirchberg88}. 
These, however, do \emph{not} induce $\Omega^K$, rather the `wrong' almost
complex structure $\Omega^\mathrm{nK}$. This means two things: first, the K\"ahler structure cannot be recovered from the two K\"ahlerian Killing spinors; secondly, it 
 reflects the fact that the Killing spinors do not define a `compatible'
$\SU(3)$-structure. For the projective space this stems from our description of 
 $\CP^3$ as $SO(5)/U(2)$, on which there is no invariant spinor 
inducing $\Omega^K$. In the other case the reason is 
that every almost Hermitian structure on the flag manifold 
is $\SU(3)$-invariant \cite{AGI98}.
\end{exa}
%
%
\subsection{Adapted connections}
Let $(M^6,g,\phis)$ be an $\SU(3)$-manifold with Levi-Civita connection $\nabla$.
As we are interested in non-integrable structures, $\nabla\phis\neq0$, we 
look for a metric connection that preserves the $\SU(3)$-structure. 
The \textit{canonical connection} defined in \eqref{nabla^n} is one such instance.\smallbreak

The space of metric connections is isomorphic to the space of $(2,1)$-tensors 
$\mathcal{A}^g:=TM^6\otimes\Lambda^2(TM^6)$ by 
$\tilde\nabla_XY=\nabla_XY+A(X,Y)$. Define the map
\bdm
\Xi:TM^6\oplus \End(TM^6)\ra\mathcal{A}^g, ~(\eta,S)\mapsto 
-S\lrcorner\psis+\frac{2}{3}\eta\otimes\omega
\edm
where $S,\eta$ are the  intrinsic tensors of the $\SU(3)$-structure on $M^6$. 
Then $\nabla^n_XY:=\nabla_XY+\Xi(\eta,S)$ 
is a metric connection on $M^6$, and we get
\bea[*]
\nabla^n_X\phis&=&\nabla_X\phis-\psis\tfrac{1}{2}(S(X),.,.)\cdot\phis
+\tfrac{1}{3}\eta(X)\omega\cdot\phis\\
&=&S(X)\cdot\phis+ \eta(X)\jphis-S(X)\cdot\phis - \eta(X)\jphis=0
\eea[*]
by Lemma \ref{lem:algrechnungendim6}, showing that $\nabla^n$ is an $\SU(3)$-connection. The space $\mathcal A^g$ splits under 
the representation of $\SO(n)$ (see \cite[p.51]{Ca25} and \cite{AF04})
into 
\bdm
\mathcal{A}^g=TM^6\oplus\Lambda^3(TM^6)\oplus\mathcal{T},
\edm
whose summands are referred to as \textit{vectorial}, \textit{skew-symmetric} 
and \textit{cyclic traceless} connections.
A computer algebra system calculates the map $\Xi$ at 
one point and gives
\begin{lem}\label{lem:connsu3}
The `pure' classes of an $\SU(3)$-manifold $M^6$ correspond to
$\nabla^n$ in:
\medskip
\begin{center}
\begin{tabular}{c|c|c|c|c|c}
class of $M^6$ &$\WS_{1\bar1}$&$\WS_{2\bar2}$&$\WS_3$&$\WS_4$&$\WS_5$\\
\hline
type of $\nabla^n$ & $\Lambda^3$ & $\mathcal{T}$ &
$\Lambda^3\oplus\mathcal{T}$ & $TM^6\oplus\Lambda^3\oplus\mathcal{T}$
&$TM^6\oplus\Lambda^3\oplus\mathcal{T}$\\
\end{tabular}
\end{center}
\medskip
\end{lem}
The projection to the skew-symmetric part of the torsion given in the previous 
lemma generates the so-called \textit{characteristic connection} $\nabla^c$.
This is a metric connection that preserves the $\SU(3)$-structure and 
additionally has the same geodesics as $\nabla$. 
If an $\SU(3)$-manifold admits such a connection, we know from \cite{FI02} 
that the $\WS_{2\bar2}$ part of the intrinsic torsion vanishes.\smallbreak

We are interested in finding out whether and when an $\SU(3)$-manifold  
$(M^6,g,\phis)$ admits a characteristic connection, that is to say when 
$$\nabla^c\psis=0.$$ 
Any connection doing that must be the (unique)  
characteristic connection of the underlying $\U(3)$-structure, so to begin with 
the $\SU(3)$-class must necessarily be $\WS_{1\bar1345}$. What is more, 
\begin{lem}\label{lem.charconn-spinor}
Given an $\SU(3)$-manifold $(M^6,g,\phis)$, a connection with skew 
torsion $\tilde\nabla$ is characteristic  if and only if it preserves 
the spinor $\phis$.
\end{lem}
\begin{proof}
Obvious, but just for the record: $\nabla^c$ is an $\SU(3)$-connection, 
and $\SU(3)=\textrm{Stab}(\phis)$ forces $\phis$ to be parallel. 
Conversely, if $\phis$ is $\tilde\nabla$-parallel, the connection must preserve 
any tensor arising in terms of the spinor, like $\omega$ and $\psis$, 
cf. Lemma \ref{lem:nablaompsip}. To conclude, just recall that the 
characteristic connection is unique (\cite{FI02}, \cite{AFH13}).
\end{proof}
To obtain the ultimate necessary \& sufficient condition we need to 
impose an additional constraint on $\WS_4, \WS_5$:
\begin{thm}\label{thm:charconnonSU}
A Riemannian  spin manifold $(M^6,g,\phis)$ admits a characteristic 
connection if and only if it is of class $\WS_{1\bar1345}$ and 
$4\eta=\delta\, \omega$. 
\end{thm}
\begin{proof}
Let $\nabla^c$ be the $\U(3)$-characteristic connection, $T$ its torsion.
We shall determine in which cases $\nabla^c\phis=\nabla^c\jphis=0$. First of all
\bea[*]
0&=&(\nabla^c_X\omega)(Y,Z)=-2(\nabla^c_X\phis,ZY\jphis)-2g(Y,Z)(\nabla^c_X\phis,\jphis).
\eea[*]
Consequently $(\nabla^c_X\phis,ZY\jphis)=0$ if $Y\perp Z$. But as  
$Y\perp Z$ vary, the spinors $YZ\jphis$ span $\phis^\perp$. In conclusion, $\nabla^c$ 
is characteristic for the $\SU(3)$-structure iff 
$(\nabla^c_X\phis,\jphis)=0$. Now choose a local adapted basis 
$e_1,\ldots ,e_6$ with $\J e_i=-e_{i+1}, i=1,3,5$. Using the formula 
\bdm
\nabla^c_X\phis=\nabla_X\phis+\frac{1}{4}(X\hook T)\phis
\edm
and $\omega(X,Y)=-(XY\phis,\jphis)$ we arrive at 
$
4\eta(X)=-(X\hook T\phis,\jphis)=\omega(X\hook T) 
=T(\omega,X)=- 1/2 \tsum T(e_i,\J e_i,X), 
$
and eventually 
$$
4\eta(X)=-\tfrac{1}{2}\tsum_{i=1}^6T(e_i,Je_i,X)=-\tsum_{i=1}^6(\nabla_{e_i}\omega)(e_i,X)
=\delta\omega(X)
$$
because $0=(\nabla^c_X\omega)(Y,Z)=(\nabla_X\omega)(Y,Z)-\frac{1}{2}(T(X,\J Y,Z)+T(X,Y,\J Z))$. 
\end{proof}
The next theorem gives an explicit formula for the torsion of 
$\nabla^c$. It relies on the computation for the Nijenhuis tensor of 
Lemma \ref{lem:Niejn}.\\
Suppose $M^6$ is of class $\WS_{1\bar1345}$, 
and decompose the intrinsic endomorphism into
$$S=\lambda \J+\mu \Id+S_{34},$$ 
as explained in Notation \ref{eq:decompS}.
\begin{thm}
Suppose $(M^6,g,\phis)$ is of class $\WS_{1\bar1345}$. Then the characteristic 
torsion of the induced $\U(3)$-structure reads 
\bdm
T(X,Y,Z)=2\lambda\psims(X,Y,Z)-2\mu\psis(X,Y,Z)-2\cyclic{XYZ}\psis(S_{34}(X),Y,Z).
\edm
If $\eta=\frac{1}{4}\delta\omega$, $T$ is the characteristic torsion 
of the $\SU(3)$-structure as well.
\end{thm}
\begin{proof}
>From Lemma \ref{lem:dom} we infer
\bdm
d\omega\circ \J(X,Y,Z)= 6\lambda\psims(X,Y,Z)-6\mu\psis(X,Y,Z)
+2\cyclic{XYZ}\psis(S_{34}(X),Y,Z)
\edm
The formula $T=N-d \omega \circ J$ (see \cite{FI02}) together with Lemma \ref{lem:Niejn} 
allows to conclude.
\end{proof}
\begin{NB}
For the class $\WS_{1\bar1}$, the torsion $T^c$ of the characteristic connection
is parallel (for nearly K\"ahler manifolds, compare \cite{Kirichenko77, AFS05}). 
For such $G$-manifolds, the $4$-form
$\sigma_T:=\frac{1}{2}\sum_i(e_i\lrcorner T)\wedge(e_i\lrcorner T)$ encodes
a lot of geometric information. It is indeed equal to $dT/2$, it measures
the non-degeneracy of the torsion, and it appears in the Bianchi identity, the
Nomizu construction, and the identity for $T^2$ in the Clifford algebra
(see \cite{Agricola&F&F13} where all these aspects are addressed).
For the class $\WS_{1\bar1}$, an easy computation shows 
\bdm
\sigma_T\ =\lambda \, d\psims-\mu\, d\psis\ =\ 12(\lambda^2+\mu^2)*\omega,
\edm
thus confirming the statement that $\sigma_T$ encodes much of the geometry:
it is basically given by the K\"ahler form.
\end{NB}
\begin{exa}\label{exa:sl2c}
Take the real $6$-manifold $M=\SL(2,\C)$ viewed as the reductive space 
$$\frac{\SL(2,\C)\times\SU(2)}{\SU(2)}=G/H$$ 
with diagonal embedding.
Let $\g, \h$ be the Lie algebras of $G$ and $H$, and set $\g=\h\oplus\m$, so that 
\bdm
\m=\{(A,B)\in\g~|~A-\bar A^t=0,\ \tr A=0,\ B+\bar B^t=0,\ \tr B=0\}.
\edm
The almost complex structure
\bdm
J(A,B)\ =\ (iA,iB)
\edm
defines a $\U(3)$-structure of class $\WS_3$, see \cite{AFS05}.
The characteristic connection $\nabla^c=\nabla+\frac{1}{2}T$ preserves a 
spinor $\phis$, so $\nabla^c$ is also characteristic for the induced 
$\SU(3)$-structure, which is of class $\WS_{35}$.
By Theorem \ref{thm:charconnonSU} we have $\eta=0$, so actually the 
$\SU(3)$-class is $\WS_3$. But then 
$\phis$ is harmonic. 
\end{exa}
The following result shows that this reflects a more general fact:
%
\begin{cor}
Whenever $\nabla^c$ exists, 
\bdm
\phis\in\Ker D \; \iff \; T\phis=0 \; \iff \;
\text{the }\SU(3)\text{-class is }\WS_3.
\edm
%
\end{cor}
\begin{proof}
By Lemma \ref{lem.charconn-spinor}, $\phis$ is $\nabla^c$-parallel;
since the Riemannian Dirac operator and the Dirac operator $D^c$ of
$\nabla^c$ are related by $D^c= D+ \frac{3}{4}T $,
the first equivalence follows. The equivalence of the first and the last
statement is a 
 direct consequence of Theorems \ref{thm:Dphisu3}, \ref{thm:charconnonSU}.
\end{proof}
Example \ref{exa:sl2c} satisfies $T\phis=0$, as shown in \cite{AFS05}, so 
the first condition should be employed if more convenient. This example
also shows that there exist 
$\SU(3)$-structures  different from type $\WS_{1\bar15}$ (namely, $\WS_3$) whose 
torsion is parallel.
%
\section{$\G$ geometry}\label{sec:g2}
%
Let $(M^7,g,\phig)$ be a Riemannian manifold with 
a globally defined unit spinor $\phig$, inducing a $\G$-structure 
$\f$ and the cross product $\times$:
\bdm
\f(X,Y,Z)\ :=\ (X Y Z \phig,\phig)\ =: \ 
g(X\times Y,Z).
\edm
We recall two standard properties (see \cite{FG82} or \cite{AH13}):
\begin{lem}\label{lem:starphiinP}
 The cross product and the $3$-form $\f$ satisfy the identities
\begin{enumerate}
\item \ $ (X\times Y)\phig=-XY\phig-g(X,Y)\phig$\smallbreak
\item \ $*\f(V,W,X,Y)=\f(V,W,X\times Y)-g(V,X)g(W,Y)+g(V,Y)g(W,X)$.
\end{enumerate}
\end{lem}
\bigbreak
Motivated by the fact that $\{X\phig~|~X\in TM^7\}=\phig^\perp$, cf. \eqref{G2-dec}, 
we have 
\begin{dfn}
There exists an endomorphism $S$ of $TM^7$ satisfying
\be\label{dfn.intrinsic-endom}
\nabla_X\phig\ =\ S(X)\phig
\ee
for every tangent vector $X$ on $M^7$, called the \emph{intrinsic endomorphism}
of $(M^7,g,\phig)$. 
\end{dfn}
%
%
\begin{lem}
The intrinsic endomorphism $S$ satisfies
\bdm
(\nabla_V\f)(X,Y,Z)=2*\f(S(V),X,Y,Z).
\edm
\end{lem}
\begin{proof}
 We calculate
\bea[*]
(\nabla_V\f)(X,Y,Z)&=&(XYZ\nabla_V\phig,\phig)+(XYZ\phig,\nabla_V\phig)\\
&=&(XYZS(V)\phig,\phig)-(S(V)XYZ\phig,\phig)\\
&=&2(XYZS(V)\phig,\phig)-2g(S(V),Z)g(X,Y)+2g(S(V),Y)g(X,Z)\\&&-2g(S(V),X)g(Y,Z).
\eea[*]
With Lemma \ref{lem:starphiinP} we get
\bea[*]
\lefteqn{2*\f(S(V),X,Y,Z)\ = } \\
&=&
-2[(XY(Z\times S(V))\phig,\phig)-g(X,Z)g(S(V),Y) + g(X,S(V))g(Y,Z)]\\
&=&2(XYZS(V)\phig,\phig)-2g(Z,S(V))g(X,Y) + 2g(X,Z)g(S(V),Y)-2g(X,S(V))g(Y,Z).
\eea[*]
\end{proof}
\begin{prop}
 The intrinsic torsion of the $\G$-structure $\f$ 
is 
\bdm
\Gamma\, =\, -\frac{2}{3}S\hook\f
\edm
where $S\hook\f(X,Y,Z)\, :=\, \f(S(X),Y,Z)$.\end{prop}
\begin{proof}
Immediate from Lemma \ref{lem:algdim7}, for 
$\frac{1}{2}\Gamma(X)\phig=\nabla_X\phig=S(X)\phig
= -\frac{1}{3}(S(X)\hook\f)\phig$. 
\end{proof}
To classify $\G$-structures one looks at endomorphisms of $\R^7$
\bdm
\End(\R^7)=\R \oplus S^2_0\R^7\oplus\g_2\oplus\R^7,
\edm
where $S^2_0\R^7$ denotes symmetric, traceless endomorphisms of $\R^7$.
The original approach to  the classification of $\G$-structures
by Fern\'andez-Gray \cite{FG82} was by the covariant derivative of the
$3$-form $\f$. In \cite{Friedrich03b} it was explained how the intrinsic torsion 
of $\G$-manifolds can be identified with $\End(\R^7)$, thus yielding an 
alternative approach to the Fern\'andez-Gray classes.  
The following result links the intrinsic
endomorphism (and thus the spinorial field equation (\ref{dfn.intrinsic-endom}))
directly to the Fern\'andez-Gray classes.
\begin{lem}\label{lem:classg2}
$\G$-structures fall into four basic types:
\medskip
\begin{center}
\begin{tabular}{|c|c|c|}
\hline
class&description& dimension\\
\hline
\hline
$\WG_1$& $S=\lambda\, \Id$&$1$\\
\hline
$\WG_2$& $S\in\g_2$&$14$\\
\hline
$\WG_3$& $S\in S^2_0\R^7$&$27$\\
\hline
$\WG_4$& $S\in \{V\lrcorner\f~|~V\in\R^7\}$&$7$\\
\hline
\end{tabular}
\end{center}
\medskip
In particular, $S$ is symmetric if and only if
 $S\in \WG_1\oplus\WG_3$ and skew iff it belongs in $\WG_2\oplus\WG_4$.
\end{lem}
%
\subsection{Spin formulation}
%
By identifying $TM^7\iso \phig^\perp$ we obtain the isomorphism 
$T^*M^7\otimes TM^7 \cong  T^*M^7\otimes\phig^\perp$,  
given explicitly by 
$$\eta\otimes X \mapsto\eta\otimes X\phig.$$ 
This enables us to describe
the tensor product directly, through $\phig$.
\smallbreak
As in the $\SU(3)$ case we will shorten $\WG_1\oplus\WG_3\oplus\WG_4$ 
to $\WG_{134}$ and so on.
%
The restricted Clifford product $\cliffmult: T^*M^7\otimes\phig^\perp\ra\Delta$ 
decomposes the space $\WG_{1234}$, as prescribed by the next result.
\begin{thm}
Let $(M^7,g,\phig)$ be a Riemannian spin manifold with unit spinor $\phi$. 
Then $\phig$ is harmonic
\bdm
D\phig=0
\edm
if and only if the underlying $\G$-structure is of class $\WG_{23}$.
\end{thm}
\begin{proof}
First of all, the spin representation splits as 
$\Delta=\R\phig\oplus \phig^\perp=\WG_{14}$, so we may write the 
intrinsic-torsion space as
\bdm
TM^7\otimes \phig^\perp= \Delta\oplus \WG_{23}.
\edm
Yet the multiplication $m$ is 
$\G$-equivariant, so 
$\Ker\cliffmult=
\{\tsum_{ij}a_{ij}e_i\otimes e_j\phig~|~(a_{ij})\in S^2_0\R^7\}=
\WG_{23}$, and the assertion follows from the definition of $D=\cliffmult\circ\nabla$.
\end{proof}
\begin{lem}\label{lem:deltapsi} 
In terms of $\phig$ the module $\WG_{24}$ depends on
\bdm
\tfrac{1}{2}\delta\f(X,Y) = (X\phig,\nabla_Y\phig)-(Y\phig,\nabla_X\phig)+(D\phig,XY\phig)+g(X,Y)(D\phig,\phig).
\edm
\end{lem}
\begin{proof}
To prove the claim we simply calculate, in some orthonormal basis $e_1,\ldots,e_7$,  
\bea[*]
\delta\f(X,Y)
&=&-\tsum(\nabla_{e_i}\f)(e_i,X,Y)
=-\tsum[(XYe_i\nabla_{e_i}\phig,\phig)+(XYe_i\phig,\nabla_{e_i}\phig)]\\
&=&- (XYD\phig,\phig)-\tsum [-2g(e_i,Y)(X\phig,\nabla_{e_i}\phig)+2g(e_i,X)(Y\phig,\nabla_{e_i}\phig)\\&&+(e_iXY\phig,\nabla_{e_i}\phig)]\\
&=&2(D\phig,XY\phig)+2(X\phig,\nabla_Y\phig)-2(Y\phig,\nabla_X\phig)+2g(X,Y)(D\phig,\phig).
\eea[*]
\end{proof}
At this point the complete picture is at hand.
\begin{thm}\label{G2thm}
The basic classes of $\G$-manifolds are described by the 
spinorial field equations for  $\phig$ as in Table \ref{table.G2}.
Here, 
$\lambda:=-\frac{1}{7}(D\phig,\phig):M\to \R$ is a real function,  $\times$ 
denotes the cross product relative to $\f$, and
\bdm
S\phig:=\sum_{i,j}g(e_i,S(e_j))e_ie_j\phig .
\edm

\begin{table}
\caption{Correspondence of $\G$-structures and spinorial field equations
(see Theorem \ref{G2thm}).}\label{table.G2}
\begin{tabular}{|c|c|}
 \hline
{class} & {spinorial equation}\\
\hline \hline
$\WG_1$& $\nabla_X\phig=\lambda X\phig$\\
\hline
$\WG_2$& $\nabla_{X\times Y}\phig=Y\nabla_X\phig-X\nabla_Y\phig+2g(Y,S(X))\phig$\\
\hline
$\WG_3$& $(X\nabla_Y\phig,\phig)=(Y\nabla_X\phig,\phig)$ and $\lambda=0$\\
\hline
$\WG_4$& $\nabla_X\phig=XV\phig+g(V,X)\phig$\quad for some $V\in TM^7$\\
\hline\hline
$\WG_{12}$& $\nabla_{X\times Y}\phig= 
  Y\nabla_X\phig-X\nabla_Y\phig+g(Y,S(X))\phig-g(X,S(Y))\phig 
  - \lambda (X\times Y)\phig$\\
\hline
$\WG_{13}$& $(X\nabla_Y\phig,\phig)=(Y\nabla_X\phig,\phig)$\\
\hline
$\WG_{14}$& $\exists V,W\in TM^7$ : $\nabla_X\phig=XVW\phig-(XVW\phig,\phig)\phig$\\
\hline
$\WG_{23}$& $S\phig=0$ and $\lambda=0$, or $D\phig=0$\\
\hline
$\WG_{24}$& $(X\nabla_Y\phig,\phig)=-(Y\nabla_X\phig,\phig)$\\
\hline
$\WG_{34}$& $3(X\phig,\nabla_Y\phig)-3(Y\phig,\nabla_X\phig)=(S\phig,XY\phig)$ and $\lambda=0$\\
\hline
$\WG_{123}$& $(S\phig,X\phig)=0$, or $D\phig=-7\lambda\phig$\\
\hline
$\WG_{124}$& $(Y\nabla_X\phig,\phig)+(X\nabla_Y\phig,\phig)=-2\lambda g(X,Y)$\\
\hline
$\WG_{134}$& $3(X\phig,\nabla_Y\phig)-3(Y\phig,\nabla_X\phig)=(S\phig,XY\phig)-7\lambda g(X,Y)$\\
\hline
$\WG_{234}$& $\lambda=0$\\
\hline
\end{tabular}
\end{table}
\end{thm}
\begin{proof}
 The proof relies on standard properties, like the fact that  
$S$ is symmetric if and only if 
$(X\nabla_Y\phig,\phig)=(Y\nabla_X\phig,\phig)$. It could be recovered by 
going through the original argument of \cite{FG82}, but we choose an 
alternative approach.\smallbreak

For $\WG_1$ there is actually nothing to prove, for the given equation 
is nothing but the Killing spinor equation characterising this type of 
manifolds \cite{FK90,BFGK91}.
\smallbreak

The endomorphism $S$ lies in $\WG_2$ if and only if 
$S(X\times Y)=S(X)\times Y + X\times S(Y)$. Then
\bea[*]
\nabla_{X\times Y}\phig&=&(-Y\times S(X)+X\times S(Y))\phig \\
&=&YS(X)\phig+g(Y,S(X))\phig-XS(Y)\phig-g(X,S(Y))\phig\\
&=&Y\nabla_X\phig-X\nabla_Y\phig+2g(S(X),Y)\phig.
\eea[*]
By taking the dot product with $\phig$ we re-obtain that $S$ is skew-symmetric.
\smallbreak

For $\WG_3$ we use Lemma \ref{lem:deltapsi}:
$$\tfrac{1}{2}\delta\f(X,Y) = 
(D\phig,XY\phig)+g(X,Y)(D\phig,\phig)+(X\phig,\nabla_Y\phig)
-(Y\phig,\nabla_X\phig).$$
This fact together with $\tr\, S = -(D\phig,\phig)$ allows to conclude. 
\smallbreak

Suppose $S \in \WG_4$. The vector representation $\R^7$ is $\{V\times . ~|~V\in \R^7\}$, so if 
$S$ is represented by $V$ we have 
$\nabla_X\phig=(V\times X)\phig =-VX\phig-g(V,X)\phig=XV\phig+g(V,X)\phig$.
\smallbreak

As for the remaining `mixed' types, we shall only prove what is not obvious.
For type  $\WG_{12}$, $S$ is of the form $S= \lambda\Id + S'$, where
$S'\in\g_2$. Thus, 
\bdm
\nabla_{X\times Y} \phig = \lambda (X\times Y)\phig + S'(X\times Y)\phig. 
\edm
The first term may be rewritten using Lemma \ref{lem:starphiinP} (1), while we deal
with the second basically as we did in the pure  $\WG_{2}$ case.
A clever rearrangement of terms then yields the desired identity.

By \cite{FG82} a structure is of type $\WG_{23}$ if $(D\phig,\phig)=0$ and 
$0=\tfrac{1}{2}\tsum_{i,j}\delta \f(e_i,e_j)\f(e_i,e_j,X)$.
This is equivalent to
\bea[*]
0&=&\tsum_{i,j}[(D\phig,e_ie_j\phig)-7g(e_i,e_j)\lambda
+(e_i\phig,S(e_j)\phig)-(e_j\phig,S(e_i)\phig)](e_ie_jX\phig,\phig)\\
&=&-\tsum_{i,j}(D\phig,e_ie_j\phig)(e_ie_j\phig,X\phig)
+2\tsum_{i,j}(e_i\phig,S(e_j)\phig)(e_ie_jX\phig,\phig).
\eea[*]
As $\{e_ie_j\phig~|~i,j=1,\ldots,7\}$ spans $\Delta$, we obtain 
$\sum_{i,j}(\phib,e_ie_j\phig)e_ie_j\phig=6\phig_1+(\phib,\phig)\phig$. Define 
\bdm
S\phig:=\sum_{i,j}g(e_i,S(e_j))e_ie_j\phig
\edm
and get $0=-6(D\phig,X\phig)
-2(S\phig,X\phig)$. Therefore $3D\phig=-S\phig$ holds on $\phig^\perp$. 
If $\lambda=0$ we then have
\bdm
(S\phig,\phig)=\tsum g(e_i,S(e_j))(e_ie_j\phig,\phig)=-\tsum g(e_i,S(e_i))=(D\phig,\phig)=0.
\edm
\bigbreak
A structure is of type $\WG_{34}$ if $(D\phig,\phig)=0$ and 
\bdm
3\delta\f(X,Y)=\tfrac{1}{2}\tsum_{i,j}\delta \f(e_i,e_j)\f(e_i,e_j,X\times Y).
\edm
Due to the calculation above, the right-hand side equals 
\bea[*]
-6(D\phig,(X\times Y)\phig)&-&2(S\phig,(X\times Y)\phig)\\
&=& 6(D\phig,XY\phig)-42g(X,Y)\lambda+2(S\phig,XY\phig)+2g(X,Y)(S\phig,\phig).
\eea[*]
As
\bdm
3\delta\f(X,Y)=6(D\phig,XY\phig)-42g(X,Y)\lambda+6(X\phig,\nabla_Y\phig)
-6(Y\phig,\nabla_X\phig),
\edm
the defining equation is equivalent to 
\bdm
3(X\phig,\nabla_Y\phig)-3(Y\phig,\nabla_X\phig)=(S\phig,XY\phig)-7g(X,Y)\lambda
\edm
and if $\lambda=0$ we get 
$3(X\phig,\nabla_Y\phig)-3(Y\phig,\nabla_X\phig)=(S\phig,XY\phig)$.
\bigbreak

As for $\WG_{124}$, note that $S$ satisfies
$(Y\nabla_X\phig,\phig)+(X\nabla_Y\phig,\phig)=-2g(X,Y)\lambda$ 
if it is skew. If symmetric, instead, it satisfies the equation iff 
$g(X,S(Y))=g(X,Y)\lambda$, i.\,e. if $S=\lambda\,\Id$.
\end{proof}
\begin{NB}
The spinorial equation for $\WG_4$ defines a connection with vectorial 
torsion \cite{AF06}. This is a $\G$-connection, since $\phig$ is parallel 
by construction.
\end{NB}
%
\subsection{Adapted connections}
%
Let $(M^7,g,\phig)$ be a $7$-dimensional spin manifold. As usual we 
identify $(3,0)$- and $(2,1)$-tensors using $g$. 
The prescription
\bdm
\nabla^n\ :=\ \nabla+\frac{2}{3}S\hook\f
\edm
defines a natural $\G$-connection, 
since 
$(X\hook\f)\phig=-3X\phig, \ 
\f\phig=7\phig$ and $\f\phib=-\phib, \forall \phib\perp\phig$. 
Abiding by Cartan's formalism, the set of metric connections is isomorphic to 
\bdm
\underbrace{\R^7}_{\Lambda^1\R^7}\oplus
\underbrace{(\R\oplus\R^7\oplus S^2_0\R^7)}_{\Lambda^3\R^7}
\oplus\underbrace{(\g_2\oplus S^2_0\R^7\oplus\R^{64})}_{\mathcal{T}}
\edm
under $\G$, and this  immediately yields  the
analogous statement to Lemma \ref{lem:connsu3}:
\begin{lem}
The `pure' classes of a $\G$-manifold $(M^7,g,\phig)$ correspond to $\nabla^n$ in:
\medskip
\begin{center}
\begin{tabular}{c|c|c|c|c}
class of $M^7$ & $\WG_1$  & $\WG_2$ & $\WG_3 $ & $\WG_4$ \\
\hline
type of $\nabla^n$ & $\Lambda^3 $ & $\mathcal{T}$ & $\Lambda^3\oplus\mathcal{T}$ &
$TM^7 \oplus\Lambda^3$ \\
\end{tabular}
\end{center}
\end{lem}
\medskip
Connections of type $\mathcal{T}$ are rarely considered, although calibrated  
$\G$-structures ($\WG_2$) have an adapted connection of this type \cite{CI07}. 
\smallbreak 

Among $\G$-connections there exists 
at most one connection $\nabla^c$ with skew-symmetric torsion $T^c$. 
Therefore we may write 
$$S=\lambda\,\Id+S_3+S_4\in\WG_{1}\oplus\WG_{3}\oplus\WG_{4}$$
with $S_3\in S_0^2TM^7$ and $S_4=V\hook \f$ for some vector $V$.
\begin{prop}
Let $(M^7,g,\phig)$ be a $\G$-manifold of type $\WG_{134}$.
The characteristic  torsion reads 
\bdm
T^c(X,Y,Z)=-\frac{1}{3}\cyclic{XYZ}\f((2\lambda\Id+9S_3+3S_4)X,Y,Z).
\edm
\end{prop}
\begin{proof}
Consider the projections
\bdm\ba{c}
T^*M^7\otimes\g_2^\perp \stackrel{\kappa}{\lra} 
\Lambda^3(T^*M^7)\stackrel{\Theta}{\lra} T^*M^7\otimes\g_2^\perp\\ 
\f(SX,Y,Z)\stackrel{\kappa}{\ltto} \frac{1}{3}\cyclic{XYZ}\f(SX,Y,Z),\qquad
T \stackrel{\Theta}{\ltto} \sum_ie_i\otimes (e_i\hook T)_{\g_2^\perp}
\ea\edm
A little computation shows that the composite $\Theta\circ\kappa$ is the identity 
map, with eigenvalues $1, 0, 2/9, 2/3$ on the four summands $\WG_i$.
But from \cite{FI02} we know that if $-2\Gamma=\Theta(T)$ for some $3$-form $T$, 
then $T$ is the  characteristic torsion.
\end{proof}
%
\section{Hypersurface theory}\label{sec:subman}
%
Let $(\bar M^7,\bar g,\phigh)$ be a $\G$-manifold and $M^6$ a 
hypersurface with transverse unit direction $V$ 
\be\label{7=6+1} 
T\bar M^7=TM^6\oplus \lan V\ran.
\ee
By restriction the spinor bundle $\bar\Sigma$ of $\bar M^7$ gives a $\Spin(6)$-bundle 
$\Sigma$ over $M^6$, and so the Clifford multiplication $\cdot$ of $M^6$ reads 
$$X\cdot \phi=VX\phi$$ 
in terms of the one on $\bar M^7$ (whose symbol we suppress, 
as usual). 
This implies, in particular, that any 
$\sigma\in\Lambda^{2k}M^6\subset \Lambda^{2k}\bar M^7$ 
of even degree will satisfy $\sigma\cdot\phi=\sigma\phi$. 
This notation was used in \cite{BGM05} to describe almost 
Killing spinors (see Section \ref{sec:KswT}). 
Caution is needed because this is not the same as  $X\cdot\phi=X\phi$ 
described in Section \ref{sec:linal} for comparing Clifford multiplications.
\smallbreak

The second fundamental form $g(W(X),Y)$ of the immersion ($W$ is the 
Weingarten map) 
accounts for the difference between the two Riemannian structures, and 
in $\bar\Sigma$ we can compare
\bdm
\bar\nabla_X\phi=\nabla_X\phi-\frac{1}{2}V W(X)\phi.
\edm
A global spinor $\phigh$ on $\bar M^7$ (a $\G$-structure) restricts to a spinor 
$\phish$ on $M^6$ (an $\SU(3)$-structure).  The next lemma explains how both the 
almost complex structure and the spin structure are, essentially, induced by 
$\phigh$ and the unit normal $V$.
\begin{lem} \label{lem:dunno}
For any section $\phib\in\Sigma$ and any vector $X\in TM^6$
\begin{enumerate}
\item\  $V\phib=\js(\phib)$\smallbreak
\item\ $V X\phish=(\J X)\phish$.
\end{enumerate}
\end{lem}
\begin{proof}
The volume form $\sigma_7$ satisfies $\sigma_7\phib=-\phib$ for 
any $\phib\in\Sigma$. 
Therefore $V\js(X\phis)=\sigma_7(X\phis)=-X\phis$.
\end{proof}
This lemma is, at the level of differential forms, 
prescribing the rule $V\hook \f =-\omega$.
\begin{prop}
With respect to decomposition \eqref{7=6+1} the intrinsic $\G$-endomorphism of 
$\bar M^7$ has the form
\be\label{eq:Sbarintermsofs}
\bar S=\begin{bmatrix}
        \J S-\frac{1}{2}\J W&*\\
	\eta &**
       \end{bmatrix}
\ee
where $(S,\eta)$ are the intrinsic tensors of $M^6$, $\J$ the almost complex 
structure, $W$ the Weingarten map of the immersion.
\end{prop}
\begin{proof}
This result was first proved in \cite{CS06} using the Cartan-K\"ahler 
machinery. Our 
argument is much simpler: the definitions imply 
$\nabla_X\phish=V S(X)\phi+\eta(X)V\phish$, and invoking 
Lemma \ref{lem:dunno}  we infer $\bar\nabla_X\phigh=\nabla_X\phigh-\tfrac{1}{2}V W(X) \phigh=
\J S(X)\phigh-\tfrac{1}{2}\J W(X) \phigh+ \eta(X)V\phigh$.
\end{proof}
The starred terms in \eqref{eq:Sbarintermsofs} should point to the half-obvious fact
that the derivative $\nabla_V\phi$ cannot be reconstructed from $S$ and $\eta$. 
As a matter of fact, later we will show that the bottom row of $\bar S$ 
is controlled by the product $(\nabla\phigh, V\phigh)$, so that the entry $**$ 
vanishes when $\nabla_V\phigh=0$.

\bigbreak
Now we are ready for the main results, which explain how to go from $M^6$ to 
$\bar M^7$ (Theorem \ref{thm:tyofsu1}) and backwards (Theorem \ref{thm:tyofsu2}). 
The run-up to those requires a preparatory definition. 
\smallbreak 

Recall that the Weingarten endomorphism $W$ is symmetric if 
the $\SU(3)$-structure is half-flat 
(Lemma \ref{lem:classsu3}). Motivated by this
\begin{dfn} 
We say that a 
hypersurface $M^6\subset\bar M^7$ has 
\begin{itemize}
 \item[(0)] \emph{type zero} if $W$ is the trivial map (meaning  
$\bar \nabla=\nabla$),\smallbreak
 \item[(I)] \emph{type one} if $W$ is of class $\WS_{\bar1}$,\smallbreak
 \item[(II)] \emph{type two} if $W$ is of class $\WS_{\bar2}$,\smallbreak
 \item[(III)] \emph{type three} if $W$ is of class $\WS_3$.
\end{itemize}
\end{dfn}
\bigbreak
Due to the freedom in choosing entries in \eqref{eq:Sbarintermsofs}, we will take 
the easiest option (probably also the most meaningful one, geometrically speaking) 
and consider only embeddings where  $\nabla_V\phigh=0$. 
\begin{thm}\label{thm:tyofsu1}
Embed $(M^6,g,\phish)$ in some $(\bar M^7,\bar g,\phigh)$ as in 
\eqref{7=6+1}, and suppose
the $\G$-structure is parallel in the normal direction: $\bar\nabla_V\phigh=0$.

Then the classes $\WG_\alpha$ 
of $(\bar M^7,\bar g,\phigh)$ depend on the column position (the class of $M^6$) 
and the row position (the Weingarten type of $M^6$) as in the table
\begin{center}
\begin{tabular}{c||c|c|c|c|c|c|c}
{} & $\WS_{1}$ & $\WS_{\bar 1}$ & $\WS_{2}$ & $\WS_{\bar 2}$ & $\WS_3$ & $\WS_4$ & $\WS_5$ \\
\hline\hline
0 & $\WG_{13}$& $\WG_4$ & $\WG_3$&$\WG_2$ &  $\WG_3$&  $\WG_{24} $&$\WG_{234}$ \\
I & $\WG_{134}$& $\WG_4$& $\WG_{34}$& $\WG_{24}$& $\WG_{34}$& $\WG_{24}$& $\WG_{234}$\\
II & $\WG_{123}$ & $\WG_{24}$ & $\WG_{23}$ & $\WG_2$& $\WG_{23}$& $\WG_{24} $& $\WG_{234}$\\
III & $\WG_{13}$ & $\WG_{34}$ & $\WG_3$ & $\WG_{23}$ & $\WG_3$ & $\WG_{234}$ & $\WG_{234}$\\
\end{tabular}
\end{center}
\end{thm}
\begin{proof}
Let $A$ be an endomorphism of $\R^6$ and $\theta$ a covector. Then 
$\bar A=\left[\begin{smallmatrix}\J A&0\\\theta&0\end{smallmatrix}\right]$ 
 is of type $\WG_4$ iff $\theta=0$ and $A$ is a multiple
of the identity, since $\J$ is given by $g(X,\J Y)=\frac{1}{2}\f(V,X,Y)$. 

With similar, easy arguments one shows that the type of 
$\bar A = \left[\begin{smallmatrix}\J A&0\\\theta&0\end{smallmatrix}\right]$ 
is determined by the class of the intrinsic tensors $(A,\theta)$ on $M^6$  
in the following way:
\begin{center}
\begin{tabular}{r||c|c|c|c|c|c|c}
$(A,\theta)\in$ & $\WS_1$ & $\WS_{\bar1}$ & $\WS_ 2$ & $\WS_{\bar2}$ & 
$\WS_ 3$ & $\WS_ 4$ & $\WS_ 5$ \\
\hline
$(\J A,\theta)\in$ & $\WS_{\bar1}$ & $\WS_1$ & $\WS_{\bar2}$ & $\WS_2$ & 
$\WS_ 3$ & $\WS_ 4$ & $\WS_ 5$\\
\hline
$\bar A\in$ & $\WG_{13}$ & $\WG_4$ & $\WG_3$ & $\WG_2$ & $\WG_3$ & $\WG_{24}$ 
& $\WG_{234}$
\end{tabular}
\end{center}
Now the theorem can be proved thus: consider for example 
 $(S,\eta)$ of class $\WS_3$ on a hypersurface of type I. 
 Then $\left[\begin{smallmatrix}\J S&0\\\eta&0\end{smallmatrix}\right]$ 
has class  $\WG_3$, and since $W$ is a multiple of the identity 
$\left[\begin{smallmatrix}\J W&0\\0&0\end{smallmatrix}\right]$ has class  
$\WG_4$. This immediately gives 
$\bar S =\left[\begin{smallmatrix}\J S-\frac{1}{2}\J W&0\\\eta 
&0\end{smallmatrix}\right]$, so the class of the $\G$-structure
 is $\WG_{34}$. All other cases are analogous.
\end{proof}
With that in place we can now do the opposite: start from the ambient 
space $(\bar M^7,\bar g,\phi)$ 
and infer the structure of its codimension-one submanifolds $M^6$. 
By inverting formula \eqref{eq:Sbarintermsofs} we immediately see 
\bdm
S=-\J\bar S\big|_{TM^6}+\frac{1}{2}W,\qq \eta(X)=g(\bar S X,V)
\edm
for any $X\in TM^6$. 
The next, final result on hypersurfaces can be found, in a 
different form, in \cite[Sect. 4]{C06}. 
\begin{thm}\label{thm:tyofsu2}
Let $(\bar M^7,\bar g,\phi)$ be a Riemannian spin manifold of 
class $\WG_\alpha$. Then a hypersurface $M^6$ with normal 
$V\in T\bar M^7$ carries an induced spin structure $\phi$: its class is an 
entry in the matrix below that is determined by the column (Weingarten 
type) and row position ($\WG_\alpha$)
\begin{center}
\begin{tabular}{c||c|c|c|c}
{} & $\WG_1$ & $\WG_2$ & $\WG_3$ & $\WG_4$ \\
\hline\hline
$0$ & $\WS_{1}$& $\WS_{\ol{1}\ol{2}45}$ & $\WS_{1235}$&$\WS_{\ol{1}45}$ \\

$I$ & $\WS_{1\ol{1}}$ & $\WS_{\ol{1}\ol{2}45}$ & $\WS_{1\ol{1}235}$ &$\WS_{\ol{1}45}$\\

$II$ & $\WS_{1\ol{2}}$& $\WS_{\ol{1}\ol{2}45}$ & $\WS_{12\ol{2}35}$& $\WS_{\ol{1}\ol{2}45}$\\

$III$ & $\WS_{13}$ &$\WS_{\ol{1}\ol{2}345}$  & $\WS_{1235}$ &$\WS_{\ol{1}345}$ \\
\end{tabular}
\end{center}
\end{thm}
\begin{proof}
In order to proceed as in Theorem \ref{thm:tyofsu1}, we prove that the class of 
an endomorphism 
$\bar A=\left[\begin{smallmatrix}\J A&*\\\theta &*\end{smallmatrix}\right]$ 
on $\R^7$ determines the class of $(A,\theta)$ on $\R^6$  in the 
following way:
\begin{center}
\begin{tabular}{r||c|c|c|c}
$\bar A\in$ & $\WG_1$ & $\WG_2$ & $\WG_3$ & $\WG_{4}$\\
\hline
$(A,\theta)\in$ & $\WS_1$ & $\WS_{\ol{1}\ol{2}45}$ & $\WS^{ }_{1235}$ & $\WS_{\ol{1}45}$
\end{tabular}
\end{center}

If $\bar A\in \WG_1$ we have $\bar A=\lambda\,\Id$ and hence 
$\theta=0$ and $A=\lambda \J$.
 
If $\bar A\in \WG_2$ then $\J A$ is skew-symmetric, and $A$ has type 
$\WS_{\ol{1}\ol{2}4}$. 

If $\bar A$ is of type $\WG_3$ it follows 
$\bar S=\left[\begin{smallmatrix}\J A&\eta\\ \eta &-\tr(\J A)
\end{smallmatrix}\right]$ 
for some symmetric $\J A$. Therefore $JA$ is of type $\WS_{\ol{1}\ol{2}3}$,  
implying the type $\WS_{123}$ for $A$.

Suppose $\bar A\in \WG_4$, so there is a vector $Z$ such that 
$g(X,\bar AY)=\f(Z,X,Y)$, whence 
\bdm
(XYZ\phi,\phi)=(\bar AY\phi,X\phi)
\edm
for every $X,Y\in \R^7$. Restrict this equation to $X,Y\in\R^6$ and put 
$Z=\lambda V+Z_1, Z_1\in\R^6$. Then $\J A=\lambda \J+A_1$ with 
$(XYZ_1\phi,\phi)=(A_1Y\phi,X\phi)$. Since $A_1$ is skew we have
\bea[*]
g(X,A_1\J Y)&=&(Z_1 X \J Y \phi,\phi)=(Z_1 X V Y \phi,\phi) 
= -(Z_1 Y V X \phi,\phi)\\
&=&-(Z_1 Y \J X \phi,\phi) =-g(Y,A_1\J X)=-g(X,\J A_1Y),
\eea[*]
so $A_1\J=-\J A_1$ and $A_1$ has type $\WS_4$. Eventually, $\J A\in\WS_{14}$.
\end{proof}
The above table explains why we cannot have a $\WG_1$-manifold 
if the derivative of $\phigh$ along $V$ vanishes. 
Moreover, in case $\nabla_V\phigh=0$ the $\WS_5$ component 
disappears everywhere, simplifying the matter a little.

Theorems \ref{thm:tyofsu1} and \ref{thm:tyofsu2} amend a petty mistake in 
\cite[Thm 3.1]{CS02} that was due to a (too) special choice of local basis.
%
\section{Spin cones}\label{sec:spincones}
%
We wish to explain how one can construct $\G$-structures, of any desired 
class, on cones over an $\SU(3)$-manifold. The recipe, which is a generalisation 
of the material presented in \cite{AH13}, goes as follows. 

As usual, start with $(M^6,g,\phisc)$ with intrinsic torsion $(S,\eta)$. 
Choose a complex-valued function $h :I\ra S^1\subset \C$ 
defined on some real interval $I$. Setting
\bdm
\phisc_t\ :=\ h(t)\phisc\ :=\  \Re h(t)\phisc + \Im h(t)\js(\phisc)
\edm
gives a new family of $\SU(3)$-structures on $M^6$ depending on 
$t\in I$, and $\jphisc_t=\js(\phisc_t)
=h(t)\jphisc$.
The product of a complex number $a\in\C$ with an 
endomorphism $A\in \End(TM)$ is defined as $aA=(\Re a)A+(\Im a)\J A$. 
Then $h(A(X)\phib)=(hA)(X)\phib=A(X)\bar h\phib$ 
for any spinor $\phib$.
The first observation is that the intrinsic torsion of $(M^6,g,\phisc_t)$ 
is given by $(h^2S,\eta)$ (cf. Remark \ref{NB:W1}(ii), with $f=h$ 
constant on $M^6$), because
\bea[*]
\nabla_X\phisc_t&=&h\nabla_X\phisc=h(S(X)\cdot\phisc)+h\eta(X)\jphisc \\
&=& (hS)(X)\cdot(\bar h h\phisc)+\eta(X)\jphisc_t
= (h^2S)(X)\cdot\phisc_t+\eta(X)\jphisc_t.
\eea[*]
If we rescale the metric conformally by some positive function 
$f:I\ra \R_+$, we may consider 
$$M^6_t:=(M^6,f(t)^2g,\phisc_t).$$ 
Note that $M^6$ and $M^6_t$ have the same Levi-Civita connection and 
spin bundle $\Sigma$, but 
distinct Clifford multiplications $\cdot\ ,\  \cdot_t\ $, albeit related by
$X\cdot\phib=\frac{1}{f(t)}X\cdot_t\phib, \forall \phib$. 
As 
\bdm
\nabla_X\phisc_t\ =\  h^2S(X)\cdot\phisc_t+\eta(X)\jphisc_t
\ = \ \tfrac{h^2}{f}S(X)\cdot_t\phisc_t+\eta(X)\jphisc_t ,
\edm
the intrinsic torsion of $M^6_t$ gets rescaled as $(\frac{h^2}{f}S,\eta)$. 
\begin{dfn}
The metric cone  
$$(\bar M^7,\bar g)=(M^6\times I,f(t)^2g+dt^2)$$
equipped with spin structure $\phigc:=\phisc_t$
 will be referred to as  the {\it spin cone} over $M^6$. 
The article \cite{AH13} considered a version of this construction 
where $f(t)=t$. 
\end{dfn}
The Levi-Civita connection $\bar\nabla^t$ of the cone reads
\bdm
\bar\nabla_XY=\nabla_XY-\frac{f'(t)}{f(t)}\bar g(X,Y)\del_t
\edm
for $X,Y\in TM^6$, whence the Weingarten map is $W=-\frac{f'}{f}\Id$. 
Furthermore, 
\bdm
\bar\nabla_{\del_t}\phigc=\bar\nabla_{\del_t}h\phisc=h'\phisc=-ih'\jphisc=-i\frac{h'}{h}hV\phisc=-i\frac{h'}{h}V\phigc.
\edm
To sum up, the intrinsic torsion of $\bar M^7$ is encoded in 
\bdm
\bar S=
\begin{bmatrix}
\frac{h^2}{f}\J S+\frac{f'}{2f}\J & 0\\ 
\eta & -i \tfrac{h'}{h} 
\end{bmatrix} .
\edm
By decomposing 
$S=\lambda \J+\mu \Id+R\in \WS_{1}\oplus\WS_{\bar1}\oplus\WS_{2\bar{2}345}$, 
the upper-left term in the matrix $\bar S$ can be written as 
\bdm
\frac{-\lambda\Im h^2+\mu \Re h^2+f'/2}{f}\J
-\frac{\lambda \Re h^2 +\mu\Im h^2}{f} \Id 
+ \frac{\Re h^2}{f}\J R-\frac{\Im h^2}{f}R.
\edm
Let us see what happens for specific choices of hypersurface structure.\\

Suppose we require $\bar M^7$ to be a nearly integrable $\G$-manifold 
(class $\WG_1$): since $\bar S$ is then 
a multiple of the identity, we need ${h'}/{h}$ to be constant, so 
$h(t)=\exp(i(ct+d)),\ c,d\in\R$. The easiest instance of this situation is the
following:

\smallskip
{\bf The sine cone}. Start with an $\SU(3)$-manifold $(M^6,g,\phisc)$ of type 
$\WS_{\bar1}$ with $S=-\frac{1}{2}\Id$. The choice $h=\textrm{e}^{it/2}$ 
produces a cone 
\bdm
(\, M^6\times (0,\pi),\, \sin(t)^2g+dt^2,\, e^{it/2}\phisc\, )
\edm
for which
$\bar S=\frac{1}{2}\Id$. This construction was 
introduced in \cite{ADHL03}, see also \cite{FIVU08, S09}.

\smallskip
{\bf Cones of pure class}. To obtain other classes of $\G$-manifolds  
we start this time by fixing the function $h=1$, so that $\phigc=\phisc$ and
\bdm
\bar S=
\begin{bmatrix}
\frac{\mu+\frac{1}{2}f'}{f}\J-\frac{\lambda}{f}\Id+\frac{1}{f}\J R & 0\\ \eta &0
\end{bmatrix},
\edm
and only now we prescribe the $\SU(3)$-structure.
\begin{itemize}
\item[a)] Take $M^6$ to be $\WS_{\ol{1}\ol{2}}$, say $S=\mu\,\Id+R$, and  $\mu<0$ constant: 
the cone 
\bdm
(\, M^6\times\R_+,\, 4\mu^2 t^2g+dt^2,\, \phisc\, )
\edm
has 
$\bar S=
\left[\begin{smallmatrix} -\frac{1}{2\mu t}\J R& 0\\ 0 &0  
\end{smallmatrix}\right] $, and so it carries a calibrated 
$\G$-structure (class $\WG_2$).\\
\item[b)] On $M^6$ of type $\WS_{\ol{1}23}$ with $\mu<0$ constant, we can build 
the same cone as in a),  
but now the resulting $\G$-structure will be balanced (class  $\WG_3$).\\
\item[c)] Take a $\WS_{\bar1}$-manifold ($S=\mu\,\Id$). 
Since $\left[\begin{smallmatrix}k(t)\J& 0\\ 0 
&0\end{smallmatrix}\right]$ is of type $\WG_4$ irrespective of the 
map $k(t)$, the cone
\bdm
(\, M^6\times I,\, f(t)^2g+dt^2,\, \phisc \, )
\edm
is always $\WG_4$, since $R$ and $\lambda$ 
vanish. When $\mu<0$, the special choice $f(t)=-2\mu t$ will 
additionally give $\bar S= 0$.
This Ansatz was used in \cite{Baer93} to manufacture a parallel $\G$-structure
(trivial class $\{0\}$) on the cone.
\end{itemize}

Other choices of $\SU(3)$-class on $M^6$ and functions $h,\, f$
will allow, along these lines, to construct any desired $\G$-class on a 
suitable cone.
%
%
\section{Killing spinors with torsion}\label{sec:KswT}
%
%
Let $(\bar M^7,\bar g,\phigk)$ be a $\G$-manifold with characteristic 
connection $\bar \nabla^c$ and torsion $\bar T$, and suppose $(M^6,g,\phisk)$ 
is a submanifold of type I or III such that $V\hook \bar T=0$, 
cf. \eqref{7=6+1}. 
The latter equation warrants that $\bar T$ restricts to a $3$-form on $M^6$; 
observe that the condition is
more restrictive than assuming $\bar\nabla_V\phi=0$, which implies only
$(V\hook \bar T) \phi=0$. \smallbreak

We decompose the Weingarten map 
$W=\mu\, \Id+W_3$ with $JW_3=-W_3J$ and prove
\begin{lem}
The differential form 
\bdm
L(X,Y,Z):=-\cyclic{XYZ}\psis(W_3(X),Y,Z)-\mu\psis(X,Y,Z)
\edm
 satisfies $(X\hook L)\phisk=-2W(X)\phisk$.
\end{lem}
\begin{proof}
In an arbitrary orthonormal basis $e_1,\ldots,e_6$ the torsion is 
$-\sum_i 
 (e_i\hook T)_{\su(3)^\perp}\otimes e_i=2\Gamma$, where 
$(e_i\hook T)_{\su(3)^\perp}$ denotes the projection of $e_i\lrcorner T$ under  
$\so(6)\ra\su(3)^\perp$. 
 It is not hard to see that the maps
\bdm\ba{c}
T^*M^6\otimes\su(3)^\perp \stackrel{\kappa}{\lra}\Lambda^3(T^*M^6)
\stackrel{\Theta}{\lra}T^*M^6\otimes\su(3)^\perp\\[1mm]
S\hook\psis-\frac{2}{3}\eta\otimes\omega\stackrel{\kappa}{\ltto}\frac{1}{3}
\cyclic\,(S\hook\psis-\frac{2}{3}\eta\otimes\omega), 
\quad T\stackrel{\Theta}{\ltto}\sum_ie_i\otimes (e_i\hook T)_{\su(3)^\perp}
\ea\edm
satisfy $\Theta\circ\kappa_{|\WS_3}=\frac{1}{3}\Id_{\WS_3}$ and 
$\Theta\circ\kappa_{|\WS_1}=\Id_{\WS_1}$. 
But since $\SU(3)$ is the stabiliser of $\phisk$, for any 
$R\in\Lambda^3T^*M^6$ we have $R(X)\,\phisk=\Theta(R)(X)\, \phisk$, so 
\bdm
(X\hook L)\,\phisk \ =\ -(\psis \hook W)\,\phisk\ =\ -2W(X)\,\phisk,
\edm
proving the lemma.
\end{proof}
For $X\in TM^6$ we have
\bdm
0\ =\ \bar\nabla_X^c\phigk=\bar \nabla_X\phigk+\tfrac{1}{4}(X\hook \bar T)\phigk
\ = \ \nabla_X\phisk+\tfrac{1}{4}(X\hook \bar T)\,\phisk
-\tfrac{1}{2}W(X)\,\phisk.
\edm
So if we define 
$$T:=\bar T_{|M^6}+L,$$ 
then $\nabla^c:=\nabla+T$ is characteristic for  $(M^6,g,\phisk)$. This 
means that if  $\bar M^7$ and $M^6$ admit characteristic connections, their 
difference must be $L$.
\begin{dfn}\label{def:gKST}
Consider the one-parameter family of metric connections 
\bdm
\nabla^s:=\nabla+2sT
\edm
passing through $\nabla^c$ at $s=1/4$ and $\nabla$ at the origin. 
A spinor $\phib$ is called 
a \textit{generalised Killing spinor with torsion} (gKST) if
\bdm
\nabla^s_X\phib=A(X)\,\phib
\edm
for some symmetric $A:TM^6\ra TM^6$. This notion captures 
many old acquaintances: taking $s=0$ will produce  
generalised Killing spinors (without torsion) \cite{BGM05, CS07}, 
and quasi-Killing spinors on Sasaki manifolds for special $A$ \cite{FK01}. 
Killing spinors with torsion correspond to $A=\Id$, $s\neq 0$ \cite{ABBK}, 
while ordinary Killing spinors arise of course from $s=0$ and $A=\Id$ 
\cite{Fr80,BFGK91}.
Our treatment intends to subsume all these notions into one and 
 shed light on the mutual relationships. 
\end{dfn}
\begin{exa}
In view of Lemma \ref{lem:classg2}, any cocalibrated $\G$-manifold (class 
$\WG_{13}$) is defined by a gKS. 
For example, the standard $\G$-structure of a $7$-dimensional $3$-Sasaki manifold 
is cocalibrated, and indeed the \textit{canonical spinor} is generalised Killing \cite{AF10}.
\end{exa}
Suppose that $\phib$, restricted to $M^6$, is a gKST. Then at any point of $M^6$
\bea[*]
\bar\nabla^s_X\phib&=&\bar\nabla_X\phib+s(X\hook \bar T)\phib
=\nabla_X\phib+s(X\hook\bar T)\phib-\tfrac{1}{2}V W(X)\phib\\
&=&\nabla^s_X\phib+s(X\hook(\bar T-T))\phib-\tfrac{1}{2}V W(X)\phib\\
&=&V (A-\tfrac{1}{2}W)(X)\phib  - s(X\hook L)\phib.
\eea[*]
Picking $A=\frac{1}{2}W$ annihilates the first term, so we are left with 
$\bar\nabla^s_X\phib=- s(X\hook L)\phib$.
Conversely, any $\bar\nabla^s$-parallel spinor on $\bar M^7$ satisfies
\bdm
0=\bar\nabla^s_X\phib\ =\ \nabla^s_X\phib+s(X\hook(\bar T-T))\,\phib
-\tfrac{1}{2}W(X)\,\phib.
\edm
To sum up,
\begin{thm}
 Let $(\bar M^7,\bar g,\phigk)$ be a $\G$-manifold with characteristic 
connection $\bar \nabla^c$ and torsion $\bar T$. Take 
a hypersurface $M^6\subset \bar M^7$ of type one or three such that 
$V\hook \bar T=0$. 
Then 
\begin{enumerate}
\item
$(M^6,g=\bar g_{|TM^6},\phigk)$ is an $\SU(3)$-manifold with characteristic connection  
$\nabla+\bar T+L$;
\item any solution $\phib$ on $\bar M^7$ to the gKST equation $
\nabla^s_X\phib=\frac{1}{2}W(X)\,\phib$ on $M^6$ must satisfy
\bdm
\bar\nabla^s_X\phib= - s(X\hook L)\phib;
\edm
\item Vice versa, if $\phib$ is $\bar\nabla^s$-parallel on $\bar M^7$,
 it solves
\bdm
\nabla^s_X\phib=-sX\hook(\bar T-T)\,\phib+\frac{1}{2}W(X)\,\phib.
\edm
\end{enumerate}
\end{thm}
%
%
\begin{exa}
Given $(M^6,g)$ we build the twisted cone
\bdm
(\bar M^7:=M^6\times\R,\bar g:=a^2t^2g +dt^2)
\edm
for some $a>0$. From the submanifold $M^6\cong M^6\times\{\frac{1}{a}\}\subset 
\bar M^7$ we can only infer the Clifford multiplication of $\bar M^7$ at  
points of $M^6\times\{\frac{1}{a}\}$. 
Therefore we consider, as in Section \ref{sec:spincones}, the hypersurface 
$M^6_t:=(M^6,a^2t^2g)\cong M^6\times\{\frac{t}{a}\}\subset 
\bar M^7$. At any point in $M^6_t$ the spinor bundles of $M^6_t$ 
and $\bar M^7$ are the same and can be identified with the spinor bundle 
of $M^6$. 
%
Hence $X\,\phib=\frac{1}{at}\del_t X\phib$. 
Since the metric of $M^6_t$ is just a rescaling of that of $M^6$, the 
Levi-Civita connections $\nabla$ coincide. For the Riemannian 
connection $\bar \nabla$ on $\bar M^7$ we have
\bdm
\bar\nabla_X\phib=\nabla_X\phib+\frac{1}{2t}\del_t X\phib=\nabla_X\phib
+\frac{a}{2}X\,\phib, 
\edm
as $W(X)=-\frac{1}{t}X$. Therefore the submanifolds $M^6_t$ are of type 
I, and one can determine the possible structures using Theorems 
\ref{thm:tyofsu1}, \ref{thm:tyofsu2}. 
Any $2$-form $\sigma$ on $M^6$ is a $2$-form on $\bar M^7$ with 
$\del_t\hook\sigma=0$, and in addition 
\bdm
\sigma\cdot\phib=a^2t^2\sigma\phib
\edm
for any spinor $\phib$.\\

Let $\phisk$ be an $\SU(3)$-structure on $M^6$ and consider the $\G$-structure 
on $\bar M^7$ given by $\phigk$. Then 
$\del_t \hook \f = -a^2t^2\omega 
$.
If $M^6$ has characteristic connection $\nabla^c$ with torsion $T$, 
\bea[*]
0&=&\nabla^c_X\phisk=\nabla_X\phisk+\tfrac{1}{4}(X\hook T)\,\phisk  = 
\bar\nabla_X\phisk-\tfrac{a}{2}X\,\phisk+\tfrac{1}{4}(X\hook T)\,\phisk\\
&=&\bar\nabla_X\phisk-\tfrac{a}{4}(X\hook\psis)\,\phisk+\tfrac{1}{4}(X\hook T)\,\phisk 
= \bar\nabla_X\phisk+\tfrac{1}{4}(X\hook(T-a\psis))\,\phisk\\
&=&\bar\nabla_X\phisk+\tfrac{1}{4}(X\hook a^2t^2(T-a\psis))\phisk,
\eea[*]
showing that $\bar T=a^2t^2(T-a\psis)$ is the characteristic torsion of $\bar M^7$.

Given a $\bar\nabla^s$-parallel spinor $\phib$ 
\bea[*]
0&=&\bar\nabla_X\phib+s(X\hook \bar T)\phib 
=\nabla_X\phib+\tfrac{a}{2}X\,\phib+\tfrac{s}{a^2t^2}(X\hook \bar T)\,\phib\\
&=&\nabla_X\phib+\tfrac{a}{2}X\,\phib+s(X\hook (T-a\psis))\,\phib 
=\nabla^s_X\phib+\tfrac{a}{2}X\,\phib-as(X\hook \psis)\,\phib,
\eea[*]
from which
\bdm
\nabla^s_X\phib-as(X\hook\psis)\,\phib=-\tfrac{a}{2}X\,\phib.
\edm
Consider the differential form on $\bar M^7$ 
\bdm
\bar \psis(X,Y,Z):=a^3t^3\psi^-(X,Y,Z) \mbox{ for } X,Y,Z\in TM^6 \mbox{ and } \del_t\lrcorner\bar\psis=0
\edm
For a Killing spinor solving $\nabla^s_X\phib=-\frac{a}{2}X\,\phib$ we then have
\bea[*]
0&=&\nabla^s_X\phib+\tfrac{a}{2}X\,\phib = 
\nabla_X\phib+\tfrac{a}{2}X\,\phib+s(X\hook T)\,\phib\\
&=&\bar\nabla_X\phib+sa^2t^2(X\hook T)\phib 
= \bar\nabla_X\phib+sa^3t^2(X\hook \psis)\phib+s(X\hook \bar T)\phib.
\eea[*]
Consequently
\bdm
0=\bar\nabla^s_X\phib+\frac{s}{t}(X\hook \bar\psis)\phib.
\edm
\end{exa}
\begin{exa}
Let $(M^7,g,\xi,\eta,\esomorphism)$ be an Einstein-Sasaki manifold with 
Killing vector $\xi$, Killing $1$-form $\eta$ and almost complex structure $\esomorphism$ 
 on $\xi^\perp$. The Tanno deformation 
($t>0$)
\bdm
g_t\ :=\ tg+(t^2-t)\eta\otimes\eta, \, \, \, \, \, \xi_t\ :=\ 
\frac{1}{t}\xi,  \, \, \, \, \,  \eta_t:=t\eta
\edm
has the property that
$(M^7,g_t,\xi_t,\eta_t,\esomorphism)$ remains Sasaki for all values of $t$. Call 
$\nabla^{g_t}$ the Levi-Civita connection of  
$(M,g_t,\xi_t,\eta_t,\esomorphism)$ and $T^{g_t}$ the characteristic torsion 
of the almost contact structure (a characteristic connection exists 
since the manifold is Sasaki). Becker-Bender proved \cite[Thm. 2.22]{BB12} 
the existence of a Killing spinor with torsion for
\bdm
\nabla^{g_t}_X+(\tfrac{1}{2t}-\tfrac{1}{2})(X\lrcorner T^{g_t}).
\edm
Quasi Killing spinors \cite{FK00} are special instances of Definition \ref{def:gKST}
 and produce gKST on the deformed Sasaki manifold
$(M^7,g_t,\xi_t,\eta_t,\esomorphism)$. 
 As proved  in \cite{BB12}, in this 
example generalised Killing spinors with torsion and Killing spinors with torsion are the same. Since the $A$ of a gKS is symmetric, the $\G$-structure given by this spinor is cocalibrated ($\WG_{13}$).
\end{exa}
\begin{exa}
In \cite{ABBK} it was proved that on a nearly K\"ahler manifold 
the sets of $\nabla^c$-parallel spinors, Riemannian Killing spinors, and
Killing spinors with torsion coincide.  
\end{exa}
To conclude, the different existing notions of (generalised) Killing spinors
(with torsion) are far from being disjoint and are best described, at least in
dimensions $6$ and $7$, using the characterising spinor
of the underlying $G$-structure as presented in this article.

\begin{NB}
At last note that the sign of the Killing constant may be reversed by choosing 
$\js(\phib)$ instead of $\phib$. 
\end{NB}
\begin{acknowledgements}
SGC was supported by an `{\sc in}d{\sc am-cofund} Fellowship in Mathematics 
and/or Applications for experienced researchers cofunded by Marie Curie 
actions', and 
thanks the FB12 at Philipps-Universit\"at Marburg for the tremendous stay.
He is indebted to Simon Salamon for sharing his insight 
at an early stage, and acknowledges his unwavering influence.

The authors thank the anonymous reviewer for reading the manuscript 
very carefully and checking almost all computations---his or her comments 
were highly appreciated. 
\end{acknowledgements}

   
\end{document}